\documentclass[11pt,a4paper]{amsart}
\usepackage{amssymb}
\usepackage{latexsym}
\usepackage{exscale}
\usepackage{amsfonts}
\usepackage{graphicx}
\usepackage{mathrsfs}
\usepackage{amsmath,amscd,amsthm}
\usepackage{bbm,pifont}
\usepackage{enumerate}
\usepackage{color}
\usepackage{amssymb}
\usepackage{latexsym}
\usepackage{exscale}
\usepackage[utf8]{inputenc}
\usepackage[T1]{fontenc}
\usepackage{dsfont}
\usepackage[mathscr]{eucal}
\usepackage{algorithm,algpseudocode}
\newcommand{\cent}{\operatorname{cent}}

\DeclareMathAlphabet{\mathpzc}{OT1}{pzc}{m}{it}

\def\ch{{\mathcal H}}

\allowdisplaybreaks

\DeclareMathAlphabet\mathcalbf{OMS}{cmsy}{b}{n}
\DeclareMathAlphabet\EuScript{U}{eus}{m}{n}
\DeclareMathAlphabet\EuScriptBold{U}{eus}{b}{n}

\usepackage[colorlinks,citecolor=red,urlcolor=blue,bookmarks=false,hypertexnames=true,backref]{hyperref}

\numberwithin{equation}{section}
\newtheorem{theorem}{Theorem}[section]
\newtheorem{lemma}[theorem]{Lemma}

\newtheorem{proposition}[theorem]{Proposition}

\newtheorem{definition}[theorem]{Definition}
\newtheorem{example}[theorem]{Example}
\newtheorem{remark}[theorem]{Remark}
\newtheorem{claim}[theorem]{Claim}

\def \Z {{\mathbb Z}}

\def \P {{\mathcal P}}

\def \z {{\bf z}}
\def \bw {{\bf w}}

\def \b0 {{\bf 0}}
\def \im {\mbox {Im}}
\def \re {\mbox {Re}}

\def \x {{\bf x}}

\def \bw {{w}}

\begin{document}
\allowdisplaybreaks

\title[Optimal lifting of Lie algebra and Cauchy--Szeg\"{o} projection on model domains]
{Optimal lifting of Levi-degenerate hypersurfaces and applications to the Cauchy--Szeg\"o projection}

\author{Der-Chen Chang,  Ji Li, Alessandro Ottazzi  and Qingyan Wu}

\address{Der-Chen Chang, Department of Mathematics and Statistics, Georgetown University,
Washington D. C. 20057, USA. \\
and \\
Graduate Institute of Business Administration, 
College of Management, Fu Jen Catholic University, 
Taipei 242, Taiwan.
}
\email {chang@georgetown.edu}

\address{Ji Li, Department of Mathematics, Macquarie University, NSW, 2109, Australia}
\email{ji.li@mq.edu.au}

\address{Alessandro Ottazzi, School of Mathematics and Statistics, University of New South Wales, Sydney 2052, Australia}
\email{a.ottazzi@unsw.edu.au}

\address{Qingyan Wu, Department of Mathematics,
         Linyi University,
         Shandong, 276005, China
         }
\email{wuqingyan@lyu.edu.cn}

\date{\today}

\begin{abstract}
We consider a family of Levi-degenerate finite type hypersurfaces in $\mathbb C^2$, where in general there is no group structure. We lift these domains to stratified Lie groups via a constructive proof, which optimizes
the well-known lifting procedure to free Lie groups of general manifolds defined by Rothschild and Stein. This yields an explicit version of the Taylor expansion with respect to the horizontal vector fields induced by the sub-Riemannian structure on these hypersurfaces. Hence, as an application, we  establish the Schatten class estimates for the commutator of the Cauchy--Szeg\"{o} projection  with respect to   a suitable quasi-metric defined on the hypersurface.
\end{abstract}

\maketitle

\section{\bf Introduction: background and main results}\label{sec:intro}

In several complex variables,  crucial objects such as the fundamental solution for the Kohn--Laplacian,
the Cauchy--Szeg\"{o} kernel, Taylor-type expansions, etc.., are described explicitly only in very few cases. However, explicit formulae are very important for related analysis, especially in unbounded domains. 
In this paper, we consider
model domains  
$\partial \Omega_k$ in $\mathbb C^2$ (we will study the higher dimensional cases in a subsequent paper), in higher step case: 
\begin{align*}
  \partial\Omega_k&:=\Big\{ (z,w)\in\mathbb C^2:\  \im(w)= \frac{1}{2k}|z|^{2k} \Big\},\quad k\geq2
\end{align*}
with the horizontal vector fields $X_1= \partial_{x_1}+x_2(x_1^2+x_2^2)^{k-1} \partial_{t}$ and
$ X_2= \partial_{x_2} -x_1(x_1^2+x_2^2)^{k-1} \partial_{t}$, where $z=x_1+ix_2$ and $t=\re(w)$. When $k=1$, $\partial \Omega_k$ is the boundary of the Siegel domain in $\mathbb C^2$ and it is CR-diffeomorphic to the first Heisenberg group, from which it inherits the group structure. However,  for $k\geq2$, $\partial \Omega_k$ does not have a group structure.
In fact, $X_1$ and $X_2$ and their Lie bracket $[X_1,X_2]$ do not generate the tangent space at every point.
In their acclaimed paper \cite{RS}, Rothschild and Stein
proved that any family of vector fields satisfying bracket generating property (\cite {CHOW, H}) can be lifted to vector fields spanning the horizontal space of a free Lie group, by adding new variables and using an approximation similar to the Euclidean approximation of differentiable manifolds. This enabled them to use geometric and analytic tools from Lie group theory in the study of 
differential operators like sum of squares of vector fields. 
In the case of our domains $ \partial\Omega_k$, $k\geq2$, Theorem 4 in \cite{RS} implies that $X_1$ and $X_2$ can be lifted to vector fields $\tilde X_1$ and $\tilde X_2$ on a higher dimensional manifold in such a way that they become free up to step $2k$ (note that $X_1$, $X_2$ and their iterated brackets of order $2k$ span the tangent space at every point). 
Moreover, Theorem 5 in \cite{RS} implies that the higher dimensional manifold
can be approximated by the free nilpotent Lie group $N$ with $2$ generators and step $2k$, in such a way  that 
$$ \tilde X_j= Y_j+R_j,\quad j=1,2, $$
where $Y_1,Y_2$ form a basis of left-invariant vector fields
generating the Lie algebra of $N$, and $R_1,R_2$ are remainders that satisfy certain conditions with respect to a suitable notion of degree.
Inspired by Rothschild and Stein's results, in our first main result we characterize explicitly the lifts of $X_1$ and $X_2$. It turns out that we may lift directly to a nilpotent Lie group $G$ of step $2k$, $k\geq2$, with dimension smaller than the dimension of the free Lie group.
Moreover,  we have that $R_j=0$, $j=1,2$. Furthermore, the method to prove our results in section 3 provides an algorithm to construct the lifting vector fields for every fixed $k$.
We now state the main result of our paper.
 
\begin{theorem}\label{main 1 Ot}
Suppose $k\geq2$. Let $\mathfrak{g}$ be the Lie algebra generated by $X_1$ and $X_2$ under the assumption $x_1^2+x_2^2\neq 0$. Let $G$ be the connected and simply connected Lie group with Lie algebra $\mathfrak{g}$, and let $n={\rm dim}(\mathfrak{g})$.
The left-invariant vector fields on $G$ 
have the form
$$
\tilde{X}_1 =  \partial_{x_1}+ \sum_{j=3}^{n-1} p_j(x_1,x_2)\partial_{x_j}       +x_2(x_1^2+x_2^2)^{k-1} \partial_{x_n},
$$
and 
$$
\tilde{X}_2 =\partial_{x_2}+ \sum_{\ell=3}^{n-1} q_\ell(x_1,x_2)\partial_{x_\ell}  -x_1(x_1^2+x_2^2)^{k-1} \partial_{x_n}
$$
for some polynomials $p_j$, $q_\ell$, and for $j,\ell=3,\dots,n-1$.
\end{theorem}

We stress that, in the case of the manifolds considered in this paper, our construction does not rely on \cite{RS}. In fact, we improve those results in two ways. 

First, the lifts in \cite{RS}  live in a free Lie algebra, which is the biggest possible nilpotent Lie algebra once a number of generators and the step are fixed. In our case, this is the free Lie algebra with two generators and step $2k$, $k\geq2$. Instead, we consider a Lie algebra with dimension given by the highest number of independent vectors that $X_1$ and $X_2$ generate, as we consider their brackets in different points. When $k=2$, the dimension of our model is $6$ against $8$ for   \cite{RS}, and
as $k$ increases our model remains considerably smaller. 

Second, in our case study, the lifts $\tilde{X}_j$ are left-invariant vector fields on a Lie group, whereas in \cite{RS}, they differ from left-invariant vector fields by the error terms $R_j$.  This remarkably reduce the difficulty when we will consider the explicit Taylor expansions of functions on $\partial \Omega_k$, $k\geq2$, via the vector fields $X_1$ and $X_2$.

We now provide the Taylor expansion as follows (for the sake of simplicity, we only state the result with remainder of second order).
\begin{theorem}\label{TaylorForOmega for Intro}
There exist positive constants $c$ and $C$ such that for all twice differentiable functions $f$  on $\partial\Omega_k$, $k\geq2$,
$$
| f({\bf y}) - P_{\bf x}({{\bf y}-{\bf x}}) | \leq C d_{cc}({\bf x},{{\bf y}})^2 \sup_{d_{cc}({\bf x},{{\bf z}})\leq c\,d_{cc}({\bf x},{{\bf y}}), i,j=1,2} | {X}_i {X}_j f({{\bf z}}) |,
$$
for all ${\bf x},{{\bf y}}\in \partial\Omega_k$, and with
$P_{\bf x}({\bf y}-{\bf x}) = f({\bf x})+({y}_1-{x}_1)X_1f({\bf x})+({y}_2-{x}_2)X_2f({\bf x})$. Here $d_{cc}$ represents the Carnot--Carath\'eodory distance on $\partial\Omega_k$. 
\end{theorem}

A direct application of Theorems \ref{main 1 Ot} and \ref{TaylorForOmega for Intro} is the Schatten class estimates for the commutator $[b,{\bf S}]$ of the Cauchy--Szeg\"o projection ${\bf S}$ on $\partial \Omega_k$.  We recall the definition of the Schatten class $S^{p}$. Note that for any compact operator $T$ on $L^{2}$,  $T^{*}T$
 is compact, symmetric and  positive. It is  diagonalizable. For $0<p<\infty$, we say that $T\in S^{p}$ if $\{\lambda_{n}\}\in \ell^{p}$, where $\{\lambda_{n}\}$ is the sequence of eigenvalues of $\sqrt{T^{*}T}$ (counted according to multiplicity).

The Schatten class  plays an important role in non-commutative analysis and complex analysis (\cite{JW, GG, LMSZ, RochSemm, LZ, Zhu}). For instance, from a spectral-theoretic perspective, the (weak) Schatten norm estimate for the commutator of a zeroth order operator with a H\"older continuous function is a first step towards showing spectral asymptotics for Hankel operators with symbols of low regularity (\cite{GG}).
Feldman and Rochberg \cite{FR} first proved  the Schatten class estimate of commutator  $[b,\mathcal C]$ of the Cauchy--Szeg\"o projection  $\mathcal C$ on the Heisenberg group $\mathbb H^n$  by using  the Cayley transform and Fourier transform.
Recently, Fan, Lacey and Li \cite{FLL} recovered this result on $\mathbb H^n$ 
via establishing a different machinery that uses Haar representation and dyadic harmonic analysis, with an approach relying on a particular dyadic 
system of cubes on $\mathbb H^n$ which is preserved under left translation. On the other hand, Lacey, Li and Wick \cite{LLW1,LLW2} developed the analytic tools to study the Schatten class in two weight settings and paraproducts in the Euclidean setting.

Returning to the model domains $\partial \Omega_k$, it is clear that the Cayley or Fourier transforms  do not apply.   Moreover, there is no group structure on $\partial \Omega_k$ for $k\geq2$. Thus, the result or proof in \cite{FLL, RochSemm} do not apply directly. 
Consider the triple $(\partial\Omega_k, d,\mu)$, where $d$ is the quasi-metric introduced in \cite{BGGr3} and $\mu$ is the Lebesgue measure on $\mathbb C\times \mathbb R$.

Our Theorems \ref{main 1 Ot} and \ref{TaylorForOmega for Intro}, together with the explicit  kernel of the Cauchy--Szeg\"o projection ${\bf S}$ (obtained in \cite{CLTW}) and the implementation of ideas of harmonic analysis techniques developed in \cite{FLL,LLW1,LLW2, RochSemm}, yield
\begin{theorem}\label{main thm} Suppose $k\geq2$. 
\begin{enumerate}
\item  Suppose that $4<p<\infty$ and $b\in  {\rm VMO}(\partial\Omega_k)$. Then   $[b,{\bf S}]\in S^p$
if and only if
 $b\in B_{p}(\partial\Omega_k)$, moreover, we have that $\|b\|_{B_{p}(\partial\Omega_k)}\approx \|[b,{\bf S}]\|_{S^p}$;

\item Suppose  $0<p\leq 4$ and $b\in C^2(\partial \Omega_k) \cap {\rm VMO}(\partial\Omega_k)$. Then $[b,{\bf S}]\in S^p$
if and only if $b$ is a constant.

\end{enumerate}
\end{theorem}
Here ${\rm VMO}(\partial\Omega_k)$ represents the standard VMO space on the space of homogeneous type $(\partial\Omega_k,d,\mu)$, and 
$B_p(\partial\Omega_k)$ is the Besov space on $\partial\Omega_k$, defined as the functions in $f\in L_{\rm loc}^{1}(\partial\Omega_k)$ such that 
\begin{align}\label{Besov norm}
\|f\|_{B_{p}(\partial\Omega_k)}=\left\{\int_{\partial\Omega_k}\int_{\partial\Omega_k}\frac{|f( \mathbf x  )-f(\mathbf y)|^p}{d(\mathbf x,\mathbf y)^{2}}d\mu(\mathbf x) d\mu(\mathbf y)\right\}^{1\over p}<\infty.
\end{align}
We note that the condition $b\in C^2(\partial \Omega_k) \cap {\rm VMO}(\partial\Omega_k)$ in (2) of Theorem~\ref{main thm}  can be weakened via the lifting and exponential map. Full details will be provided in Section 6.

We compare Theorem~\ref{main thm}  with  previous closely related results \cite{FLL, FR, JW, LLW1, LLW2, RochSemm}. The underlying space and the singular integrals in those results are homogeneous in the sense that the measure of the ball is comparable
to its radius  to the power $n$, i.e., $\mu(B(x,r))\approx r^n$, and the size of the kernel $K(x,y)$ is comparable to $d(x,y)^{-n}$. Hence, in earlier results, the critical index $p$, below which the Besov space contains only constants, equals the homogeneous dimension $n$. 
In our Theorem \ref{main thm}, the underlying space  $(\partial\Omega_k, d,\mu)$ does not have a homogeneous dimension. Rather, the measure of the ball is controlled by $r^{2k+2}$ from above (upper dimension) and by $r^4$ from below (lower dimension). We show that the critical index is the lower dimension $4$.

We stress that our lifting theorem with the consequent result on Taylor polynomials  has further applications to the multiparameter settings. Namely, to the Shilov boundaries in $\mathbb C^{2}\times\cdots\times \mathbb C^{2}$ studied by Nagel--Stein in \cite{NS2004}, which encompass a collection of $n$ model domains $\widetilde M=\partial\Omega_{k_1}\times\cdots\times \partial\Omega_{k_n}$, $k_i\geq2$, $i=1,\ldots,n$.  We will study this case in a forthcoming paper.

\medskip
\section{\bf Preliminaries on the model domain $(\partial\Omega_k, d,\mu)$}

Suppose $k\geq2$. Consider
\begin{align*}
  \partial\Omega_k&:=\left\{ (z,w)\in\mathbb C^2:\  \im(w)= \frac{1}{2k}|z|^{2k} \right\}
\end{align*}
with $z=x_1+ix_2$ and $t=\re(w)$. For points in $ \partial\Omega_k$ we will interchange the several notations ${\bf x}=(z,t)=(x_1,x_2,t)$. We will also need to embed $\partial\Omega_k$ into $\mathbb R^N$ for a suitable $N$. Hence, in some instances the points will be represented by ${\bf x}=(x_1,x_2,x_N)$.

We now recall the quasi-distance $d$ on $\partial\Omega_k$ (\cite{BGGr3}) , $k\geq2$, as follows: for every $ (z, t),(\bw, s) \in \partial \Omega_k$,
\begin{align}\label{dist}
d(\color{black}{ (z, t),(\bw, s) }):=h^2(\color{black}{ (z, t),(\bw, s) })\rho^{2-2k}(\color{black}{ (z, t),(\bw, s) }),
\end{align}
where
\begin{align}\label{rho}
\rho(\color{black}{ (z, t),(\bw, s) ):=  |z|+|\bw|+|\sigma|^{1\over2k}\approx  |z|+|\bw|+|t-s|^{1\over2k}
},
\end{align}
and
\begin{align}\label{function h}
h(\color{black}{ (z, t),(\bw, s) })=|z-\bw|^2\rho^{2k-2}(\color{black}{ (z, t),(\bw, s) })+|\sigma(\color{black}{ (z, t),(\bw, s) })|
\end{align}
with
$$ \sigma( (z, t),(\bw, s) ) =  t-s+2\im (z^k\overline{\bw}^k).$$
Based on Proposition 9.6 in \cite{BGGr3}, we see that this quasi-metric $d$ satisfies:
\begin{equation}\label{cd}
 d((z, t),(z', t')) \leq C_d\left( d((z, t),(\bw, s))+ d((\bw, s),(z', t'))\right)
 \end{equation}
with some constant $C_d>1$.

Define the $d$-metric ball as 
$$B\left((z,t),r\right)=\{(w,s): d((z,t),(\bw, s))<r\},\quad (z,t)\in \mathbb C\times\mathbb R.$$
According to \cite[Proposition 9.8] {BGGr3}, the measure of the ball $B\left((z,t),r\right)$ satisfies
\begin{equation}\label{mball}
\mu(B\left((z,t),r\right))\approx r.
\end{equation}
Thus, $(\partial\Omega_k, d,\mu)$ forms a space of homogeneous type in the sense of Coifman and Weiss. For notational convenience, we use $|B|$ to denote the measure of $B$, that is,
$$ |B\left((z,t),r\right)|=\mu(B\left((z,t),r\right)). $$ 
Furthermore, we will simply write  $d{\bf x}$ and $d{\bf y}$ instead of $d\mu(\mathbf x)$ and $d\mu(\mathbf y)$.

\subsection{\bf The Cauchy--Szeg\"o kernel on $\partial\Omega_k$}

We recall that in the very recent result of Chang, Li, Tie and Wu \cite{CLTW}, they 
obtained the explicit pointwise size estimate and regularity estimate of the Cauchy--Szeg\"o kernel on $\partial\Omega_k$.

\begin{theorem}[\cite{CLTW}]\label{thm CLTW}
For $(z,t)$ and $(\bw, s)$ in $\partial \Omega_k$ with $(z,t)\not=(\bw, s)$,  the Cauchy--Szeg\"o projection associated with the kernel  $S(z,t;\bw, s)$ is a Calder\'on--Zygmund operator on $(\partial\Omega_k, d,\mu)$, i.e.,
\begin{align}\label{k-size1}
|S(z,t;\bw, s)|
={1\over d((z,t),(\bw, s))}. 
\end{align}
For $(z,t)\not=(\bw_0, s_0)$ and
for $ d((\bw_1, s_1),  (\bw_0, s_0) )\leq  c d((z,t),(\bw_0, s_0)),$
with some small $c$,
\begin{align} \label{k-size2}
|S(z,t;\bw_1, s_1)-S(z,t;\bw_0, s_0)|
 \leq C_1  { 1\over d((z,t),(\bw_0, s_0)) } \ \bigg({d( (\bw_1,s_1),(\bw_0,s_0) )\over d((z,t),(\bw_0, s_0))} \bigg) ^{1\over2k+2}  
 \end{align}
with some constant $C_1>0$.
\end{theorem}

Let $\rho$ and $h$ be the same as in \eqref{rho} and \eqref{function h}, respectively.
We also define the auxiliary functions: 
\begin{align*}
&A(z,t;\bw, s)= {1\over 2}\big(|z|^{2k}+|\bw|^{2k} -i(t-s)\big);\quad
\P(z,t;\bw, s)={{z \overline{\bw}\over A(z,t;\bw, s)^{1\over k}}}.
\end{align*}
By Lemma 9.3 in \cite{BGGr3} and the estimate on Page 242 in \cite{BGGr4}, we have the following properties.
\begin{lemma}\label{lem6.1}
The functions $h, \rho, A$ and $\P$ satisfy
\begin{equation}\begin{split}\label{est 1}
& |z^k-\bw^k|^2\lesssim h ((z, t),(\bw, s))\lesssim \rho^{2k} ((z, t),(\bw, s)) \approx |A(z,t;\bw, s)|; \\[6pt]
& | 1-P(z,t;\bw, s) |\approx {h ((z, t),(\bw, s))\over |A(z,t;\bw, s)|},\quad (z,t), (\bw, s)\in\partial\Omega_k. 
\end{split}\end{equation}
\end{lemma}
Based on these notation, we see that the Cauchy--Szeg\"o  kernel  $S(z,t;\bw, s)$ on $\partial \Omega_k$ can be expressed as 
\begin{align}
S(z,t;\bw, s) = {1\over 4\pi^2}A^{-{k+1\over k}}(z,t;\bw, s) \big(1-\P(z,t;\bw, s)\big)^{-2}.
\end{align}
Moreover, based on \eqref{est 1}, we see that
\begin{align}\label{est33}
d((z,t),(\bw, s)) \lesssim \rho^{2k+2} ((z, t),(\bw, s)).
\end{align}

\smallskip

\subsection{Dyadic systems on $(\partial\Omega_k, d,\mu)$}\label{sec:dyadic_cubes}

Note that  $(\partial\Omega_k, d,\mu)$ as mentioned in Section \ref{sec:intro} is a space of homogeneous type in the sense of Coifman and Weiss.

A
countable family
$\mathscr{D}
    := \cup_{h\in\Z}\mathscr{D}_h$, with
    $\mathscr{D}_h
    :=\{Q^h_\alpha\colon \alpha\in \mathscr{A}_h\}
$
and Borel sets $Q^h_\alpha\subseteq {\partial\Omega_k}$, is called \textit{a
system of dyadic cubes with parameters} $\delta\in (0,1)$ and
$0<a_1\leq A_1<\infty$ if it has the following properties:
\begin{align} 
 (i)\hskip0.2cm& {\partial\Omega_k}= \bigcup_{\alpha\in \mathscr{A}_h} Q^h_{\alpha}
    \quad\text{(disjoint union) for all}~h\in\Z; \nonumber
\\
 (ii)\hskip0.2cm&   \text{if }\ell\geq h\text{, then either }
        Q^{\ell}_{\beta}\subseteq Q^h_{\alpha}\text{ or }
        Q^h_{\alpha}\cap Q^{\ell}_{\beta}=\emptyset;\nonumber\\
  (iii)\hskip0.2cm&  \text{for each }(h,\alpha)\text{ and each } \ell\leq h,
    \text{ there exists a unique } \beta
    \text{ such that }Q^{h}_{\alpha}\subseteq Q^\ell_{\beta};
\nonumber \\
(iv)\hskip0.2cm& \text{for each $(h,\alpha)$ there exists at most $M$ 
        (a fixed geometric constant)  numbers $\beta$ such that } \nonumber\\ 
    &  Q^{h+1}_{\beta}\subseteq Q^h_{\alpha}, \hskip0.2cm\text{ and }
        \hskip0.2cm Q^h_\alpha =\bigcup_{{Q\in\mathscr{D}_{h+1},
    Q\subseteq Q^h_{\alpha}}}Q;
\nonumber \\
(v)\hskip0.2cm&  B({\bf x}^h_{\alpha},a_1\delta^h)
    \subseteq Q^h_{\alpha}\subseteq B({\bf x}^h_{\alpha},A_1\delta^h)
    =: B(Q^h_{\alpha});\nonumber
\\
 (vi)\hskip0.2cm&  \text{if }\ell\geq h\text{ and }
   Q^{\ell}_{\beta}\subseteq Q^h_{\alpha}\text{, then }
   B(Q^{\ell}_{\beta})\subseteq B(Q^h_{\alpha}). \nonumber 
\end{align}
The set $Q^k_\alpha$ is called a \textit{dyadic cube of
generation} $k$ with centre point ${\bf x}^k_\alpha\in Q^k_\alpha$
and sidelength~$\delta^k$.

From the properties of the dyadic system above and from the doubling measure, we can deduce that there exists a constant
$C_{\mu,0}$ depending only on the constants as in \eqref{mball} and $a_1, A_1$ as above, such that for any $Q^k_\alpha$ and $Q^{k+1}_\beta$  with $Q^{k+1}_\beta\subset Q^k_\alpha$,
\begin{align}\label{Cmu0}
|Q^{k+1}_\beta|\leq |Q^k_\alpha|\leq C_{\mu,0}|Q^{k+1}_\beta|.
\end{align}

We recall from \cite{HK} the following construction. 

\begin{theorem}\label{theorem dyadic cubes}
On  $(\partial\Omega_k, d,\mu)$, there exists a system of dyadic cubes with parameters
$0<\delta\leq (12A_0^3)^{-1}$ and $a_1:=(3A_0^2)^{-1},
A_1:=2A_0$. The construction only depends on some fixed set of
countably many centre points ${\bf x}^k_\alpha$, having the
properties that
   $ d({\bf x}_{\alpha}^k,{\bf x}_{\beta}^k)
        \geq \delta^k$ with $\alpha\neq\beta$,
    $\min_{\alpha}d({\bf x},{\bf x}^k_{\alpha})
        < \delta^k$ for all ${\bf x}\in \partial\Omega_k,$
   and a certain partial order ``$\leq$'' among their index pairs
$(k,\alpha)$. In fact, this system can be constructed in such a
way that
$$
    \overline{Q}^k_\alpha
        =\overline{\{{\bf x}^{\ell}_\beta:(\ell,\beta)\leq(k,\alpha)\}}, \quad\quad
    \widetilde{Q}^k_\alpha:=\operatorname{int}\overline{Q}^k_\alpha=
    \Big(\bigcup_{\gamma\neq\alpha}\overline{Q}^k_\gamma\Big)^c,
\quad \quad   \widetilde{Q}^k_\alpha\subseteq Q^k_\alpha\subseteq 
    \overline{Q}^k_\alpha,
$$
where $Q^k_\alpha$ are obtained from the closed sets
$\overline{Q}^k_\alpha$ and the open sets $\widetilde{Q}^k_\alpha$ by
finitely many set operations.
\end{theorem}

We also recall the following remark from \cite[Section 2.3]{KLPW}.
The construction of dyadic cubes requires their centre points
and an associated partial order be fixed \textit{a priori}.
However, if either the centre points or the partial order is
not given, their existence already follows from the
assumptions; any given system of points and partial order can
be used as a starting point. Moreover, if we are allowed to
choose the centre points for the cubes, the collection can be
chosen to satisfy the additional property that a fixed point
becomes a centre point at \textit{all levels}:
\begin{equation}\label{eq:fixedpoint}
\begin{split}
    &\text{given a fixed point } {\bf x}_0\in X, \text{ for every } k\in{\mathbb Z},
        \text{ there exists }\alpha \text{ such that } \\
    & {\bf x}_0
        = {\bf x}^k_\alpha,\text{ the centre point of }
        Q^k_\alpha\in\mathscr{D}_k.
\end{split}
\end{equation}

\subsection{Adjacent Systems of Dyadic Cubes}

On $(\partial\Omega_k, d,\mu)$, a
finite collection $\{\mathscr{D}^{\mathfrak t}\colon {\mathfrak t}=1,2,\ldots ,\mathpzc T\}$ of the dyadic
families  is called a \textit{collection of
adjacent systems of dyadic cubes with parameters} $\delta\in
(0,1), 0<a_1\leq A_1<\infty$ and $1\leq C_{adj}<\infty$ if it has the
following properties: individually, each $\mathscr{D}^{\mathfrak t}$ is a
system of dyadic cubes with parameters $\delta\in (0,1)$ and $0
< a_1 \leq A_1 < \infty$; collectively, for each ball
$B({\bf x},r)\subseteq \partial\Omega_k$ with $\delta^{k+3}<r\leq\delta^{k+2},
k\in\mathbb Z$, there exist ${\mathfrak t} \in \{1, 2, \ldots, \mathpzc T\}$ and
$Q\in\mathscr{D}^{\mathfrak t}$ of generation $k$ and with centre point
${}^{\mathfrak t}{\bf x}^k_\alpha$ such that $d({\bf x},{}^{\mathfrak t}{\bf x}_\alpha^k) <
2A_0\delta^{k}$ and
\begin{equation}\label{eq:ball;included}
    B({\bf x},r)\subseteq Q\subseteq B({\bf x},C_{adj}r).
\end{equation}

We apply \cite{HK} to $(\partial\Omega_k, d,\mu)$ to get the following construction.

\begin{theorem}\label{thm:existence2}
    $(\partial\Omega_k, d,\mu)$ is a  space of homogeneous type.
    There exists a collection $\{\mathscr{D}^{\mathfrak t}\colon
    {\mathfrak t} = 1,2,\ldots ,\mathpzc T\}$ of adjacent systems of dyadic cubes with
    parameters $\delta\in (0, (96C_d^6)^{-1}), a_1 := (12C_d^4)^{-1},
    A_1 := 4C_d^2$ and $C := 8C_d^3\delta^{-3}$. The centre points
    ${}^{\mathfrak t}{\bf x}^k_\alpha$ of the cubes $Q\in\mathscr{D}^{\mathfrak t}_k$ have, for each
    ${\mathfrak t}\in\{1,2,\ldots,\mathpzc T\}$, the two properties
    \begin{equation*}
        d({}^{\mathfrak t}{\bf x}_{\alpha}^k, {}^{\mathfrak t}{\bf x}_{\beta}^k)
        \geq (4C_d^2)^{-1}\delta^k\quad(\alpha\neq\beta),\qquad
        \min_{\alpha}d({\bf x},{}^{\mathfrak t}{\bf x}^k_{\alpha})
        < 2C_d\delta^k\quad \text{for all}~{\bf x}\in \partial\Omega_k.
    \end{equation*}
    Moreover, these adjacent systems can be constructed in such a
    way that each $\mathscr{D}^{\mathfrak t}$ satisfies the distinguished
    centre point property \eqref{eq:fixedpoint}.
\end{theorem}

We recall from \cite[Remark 2.8]{KLPW}
that the number $\mathpzc T$ of the adjacent systems of dyadic
    cubes as in the theorem above satisfies the estimate
    \begin{equation*}
        \mathpzc T
        = \mathpzc T(C_d,\widetilde A_1,\delta)
        \leq \widetilde A_1^6(C_d^4/\delta)^{\log_2\widetilde A_1},
    \end{equation*}
where $\widetilde A_1$ is the geometrically doubling constant, see \cite[Section 2]{KLPW}.

\subsection{An Explicit Haar Basis on Spaces of Homogeneous Type}

Next we recall the explicit construction in \cite{KLPW} of a Haar basis
$\{h_{Q}^{\epsilon}: Q\in \mathscr{D}, \epsilon = 1,\dots,M_Q - 1\}$ and adapt the notation to  $L^p(\partial \Omega_k,\mu)$,
$1 < p < \infty$, associated to the dyadic cubes
$Q\in\mathscr{D}$ as follows. Here $M_Q := \#\ch(Q) = \# \{R\in
\mathscr{D}_{k+1}\colon R\subseteq Q\}$ denotes the number of
dyadic sub-cubes (``children'') the cube $Q\in \mathscr{D}_k$
has; namely $\mathcal{H}(Q)$ is the collection of dyadic children of $Q$.

\begin{theorem}[\cite{KLPW}]\label{thm:convergence}
    For $1 < p < \infty$, for each $f\in
    L^p(\partial\Omega_k,\mu)$, we have
    \[
        f(x)
        =  \sum_{Q\in\mathscr{D}}\sum_{\epsilon=1}^{M_Q-1}
            \langle f,h^\epsilon_Q\rangle h^\epsilon_Q(x), 
    \]
    where the sum converges $($unconditionally$)$ both in the
    $L^p(\partial \Omega_k,\mu)$-norm and pointwise $\mu$-almost everywhere.
 \end{theorem}

\subsection{Characterization of Schatten Class}
In 1989, Rochberg and Semmes \cite{RochSemm} introduced the notion
of nearly weakly orthogonal ({\rm NWO}) sequences of functions. 
\begin{definition}
	Let $\{e_{Q}\}_{Q\in \mathscr{D}}$ be a collection of functions.  We say $\{e_Q\}_{Q\in \mathscr{D}}$ is {\rm NWO} sequences, if  {\rm supp}$\{e_Q\}\subset Q$ and the maximal function $f^*$ is bounded on $L^{p}(\partial\Omega_k,\mu)$, where $f^*$ is defined as
	$$f^*(x)=\sup_{Q}\frac{|\langle f,e_Q \rangle|}{|Q|^{1/2}}\chi_{Q}(x).$$
\end{definition}

We will use the following result proved by Rochberg and Semmes.
\begin{lemma}[\cite{RochSemm}]
If the collection of functions $\left\{e_Q: Q \in \mathscr{D}\right\}$ are supported on $Q$ and satisfy for some $2<p<\infty,\left\|e_Q\right\|_{L^p(\partial \Omega_k,\mu)} \lesssim|Q|^{1 / p-1 / 2}$, then $\left\{e_Q\right\}_{Q\in \mathscr{D}}$ is {\rm NWO} sequences.
\end{lemma}
If some operator $T$ belongs to $S^{p,q}(L^2(\partial \Omega_k,\mu))$, 
Rochberg and Semmes developed a substitute for the Schmidt decomposition of the operator $T$ that will be representations of the form
\begin{align}
T=\sum_{Q\in \mathscr{D}}\lambda_{Q}\langle\cdot,e_{Q}\rangle f_{Q}
\end{align}
with $\{e_{Q}\}_{Q\in \mathscr{D}}$ and $\{f_{Q}\}_{Q\in \mathscr{D}}$ are {\rm NWO} sequences and $\{\lambda_{Q}\}_{Q\in \mathscr{D}}$ is a sequence of scalars. It is easy to see that
\begin{align}\label{eq-NWO1}
\|T\|_{S^{p,q}(L^2(\partial \Omega_k,\mu))}\lesssim \|\lambda_Q\|_{\ell^{p,q}}, \quad 0<p<\infty, 0<q<\infty.
\end{align}
When $1<p=q<\infty$, Rochberg and Semmes also obtained
\begin{lemma}[\cite{RochSemm}]\label{eq-NWO}
For any bounded compact operator $T$ on $L^2(\partial\Omega_k,\mu)$ and $\{e_{Q}\}_{Q\in \mathscr{D}}$ and $\{f_{Q}\}_{Q\in \mathscr{D}}$ are {\rm NWO} sequences, then for $1<p<\infty$, 
$$\bigg[\sum_{Q\in \mathscr{D}}\left|\langle T e_{Q},f_{Q}\rangle\right|^{p}\bigg]^{\frac{1}{p}}\lesssim \|T\|_{S^{p}(L^2(\partial \Omega_k,\mu))}.$$
\end{lemma}

\smallskip

\section{Lifting $\partial\Omega_k$ to a nilpotent Lie group $G$ and Taylor expansion}\label{sec:lifting}

The main purpose of this section is to define a nilpotent Lie group $G$ onto which we will embed the manifold $\partial\Omega_k$. We will show that the horizontal vector fields defining the CR structure on $\partial\Omega_k$ can be lifted to horizontal, left-invariant vector fields on such group $G$. We will give a somewhat explicit characterisation of these vector fields, for which it is convenient to introduce suitable exponential coordinates of the second kind. This is the content of our first main theorem. Finally, we will recall the Taylor polynomial theorem for functions on $G$ proved in \cite{FS} and adapt it to our coordinates.
Throughout all this section, we will use classical properties of Lie groups and nilpotent Lie groups. For an insight see, e.g., \cite{Knapp}.

\subsection{The lifting group $G$}
For $k\geq 2$, consider the vector fields
$$ X_1= \partial_{x_1}+x_2(x_1^2+x_2^2)^{k-1} \partial_{t}, $$
$$ X_2= \partial_{x_2} -x_1(x_1^2+x_2^2)^{k-1} \partial_{t}. $$
Then the complex vector field $L=X_1+iX_2$ defines a CR-structure on the manifold
$$
\partial\Omega_k= \bigg\{(z,w)\in {\mathbb C}^2: {\rm Im}(w) = \frac{1}{2k}(x_1^2+x_2^2)^{k}\bigg\},
$$
where $z=x_1+ix_2$ and ${\rm Re}(w)= t$.
We are going to lift $X_1$ and $X_2$ to left-invariant vector fields $\tilde{X}_1$ and $\tilde{X}_2$ on a suitable stratified Lie group. We construct this Lie group by means of its Lie algebra. Assume 
that $x^2_1+x^2_2\neq 0$, and let $\mathfrak{g}$ be the Lie algebra generated by the vector fields $X_1$ and $X_2$ in this case. We recall that a stratified Lie algebra of step $s$ is a nilpotent Lie algebra such that
$$
\mathfrak{g} = \sum_{j=1}^s \mathfrak{g}_j,
$$
a vector space direct sum with $[\mathfrak{g}_1,\mathfrak{g}_j]=\mathfrak{g}_{j+1}$ for all $j=1,\dots,s-1$.
\begin{lemma}
The Lie algebra $\mathfrak{g}$ is nilpotent of step $2k$ and stratified. The center of $\mathfrak{g}$ is ${\rm span}\{\partial_{t}\}$. Moreover, if 
$$
\mathfrak{g} = \sum_{j=1}^{2k} \mathfrak{g}_j,
$$
is a stratification,
then $\sum_{j=2}^{2k-1} \mathfrak{g}_j$ is a sub-algebra of $\mathfrak{g}$.
\end{lemma}
\begin{proof}
The proof follows from considering subsequent brackets between the vector fields $X_1$ and $X_2$.
\end{proof}
Let ${\rm dim}({\mathfrak g})=N$.
In what follows, we shall treat $\mathfrak{g}$ as an abstract Lie algebra, denoting $X_1$, $X_2$, and $\partial_{t}$  by $e_1$, $e_2$, and $e_N$, respectively. Set $\mathfrak{g}_1={\rm span}\{e_1,e_2\}$ and $\mathfrak{g}_{2k}={\rm span}\{e_N\}$. Complete $e_1$, $e_2$, and $e_N$ to a basis $\{e_1,\dots,e_N\}$ of $\mathfrak{g}$ that respects the stratification, i.e., $\{e_3\}$ is a basis of $\mathfrak g_2$, $\{e_4,e_5\}$ is a basis of $\mathfrak g_2$, and so on. The Lie algebra $\mathfrak{g}$ admits a family of dilations $\{\delta_{\lambda}\,:\,\lambda >0\}$. These are automorphisms of $\mathfrak{g}$ defined by $\delta_\lambda(e_\ell)= \lambda^{j} e_\ell$ for every $e_\ell \in \mathfrak{g}_j$ and every $j=1,\dots,2k$.

By Lie’s third fundamental theorem, there is a unique connected and simply connected Lie group $G$ whose Lie algebra is $\mathfrak{g}$. 
The family of automorphisms $\delta_{\lambda}$ defines  a family of automorphisms on $G$ whose derivatives are $\delta_{\lambda}$, and for which we use the same notation. We denote by ${\mathcal Q}= \sum_{j=1}^{2k}j{\rm dim}(\mathfrak{g}_j)$
 the homogeneous dimension of $G$.

Since $G$ is nilpotent, the exponential map $\exp: \mathfrak{g} \to G$ is a diffeomorphism. 
In our set up it is convenient to use exponential coordinates of the second kind.  The mapping $\Phi : {\mathfrak g}\cong {\mathbb R}^N \to G$ defined by
$$
{\boldsymbol \xi} = (x_1,x_2,x_3,\dots,x_{N-1},x_N)\mapsto \exp\left(\sum_{j=3}^{N-1}x_je_j\right)\exp\left(x_1e_1+x_2e_2\right)
\exp\left(x_N e_N\right)$$
is a global  diffeomorphism. Let $(\mathbb R^N, \ast)$ be the group with
$$
{\boldsymbol \xi}\ast{\boldsymbol \eta} = \Phi^{-1} \left(\Phi({\boldsymbol \xi})\Phi({\boldsymbol \eta})\right)
$$
for every ${\boldsymbol \xi},{\boldsymbol \eta}\in \mathbb R^N$. Note that the mapping
$$
\Phi : (\mathbb R^N, \ast) \to (G, \cdot)
$$
is an isomorphism of Lie groups. If ${\boldsymbol \xi}\in (\mathbb R^N, \ast)$, then 
$
\left(\delta_\lambda({\boldsymbol \xi}) \right)_{\ell}= \lambda^j x_\ell,
$
if $e_\ell\in \mathfrak{g}_j$.
Moreover, there exists a global diffeomorphism $\Psi = (\psi_1,\dots,\psi_N)$ of $\mathbb R^N$ such that 
\begin{equation}\label{PsiDiff}
\Phi({\boldsymbol \xi})= \exp\left(\sum_{j=3}^{N-1}x_je_j\right)\exp\left(x_1e_1+x_2e_2\right)\exp\left(x_Ne_N\right) = \exp\left( \sum_{j=1}^{N}\psi_j({\boldsymbol \xi})e_j\right).
\end{equation}
Let $\tilde{X}_1$ and $\tilde{X}_2$ be the left-invariant vector fields on $(\mathbb R^N, \ast)$ that coincide with $e_1$ and $e_2$ at the identity. 
For every ${\boldsymbol \xi}\in \mathbb R^N$, define the horizontal space ${\mathcal H}_{{\boldsymbol \xi}}={\rm span}\{\tilde{X}_1({{\boldsymbol \xi}}), \tilde{X}_2({{\boldsymbol \xi}})\}$. We say that an absolutely continuous curve $\gamma: [0,\rho]\to (\mathbb R^N, \ast)$ is horizontal if $\dot\gamma(s) \in {\mathcal H}_{\gamma(s)}$ for a.e. $s\in [0,\rho]$. The Carnot--Carath\'eodory distance is then defined by 
\begin{align*}
\tilde{d}_{cc}({\boldsymbol \xi},{\boldsymbol \eta}) &= \inf\{\rho: \exists\ {\rm horizontal \ curve\ }\ \gamma: [0,\rho]\to (\mathbb R^N, \ast) \ {\rm s.t.\ } \ \|\dot\gamma(s)\|_{\tilde{cc}}\leq1\\
 &\hskip2cm  {\rm\ for\ } s\in[0,\rho], \gamma(0)={\boldsymbol \xi}, \gamma(\rho)={\boldsymbol \eta}\},
\end{align*}
where
$$ \|c_1{\tilde X}_1+c_2{\tilde X}_2\|_{\tilde{cc}}:= \sqrt{c_1^2+c_2^2}.  $$ 
The distance satisfies $\tilde{d}_{cc}(\delta_\lambda({\boldsymbol \xi}),\delta_\lambda({\boldsymbol \eta})) =\lambda \,\tilde{d}_{cc}({\boldsymbol \xi},{\boldsymbol \eta}) $.
Note that the differential $\Phi_*$ maps $\tilde{X}_j$ to $\bar{X}_j$, $j=1,2$, where $\bar{X}_1$ and $\bar{X}_2$ are the left-invariant vector fields on $G$ computed using exponential coordinates of the first kind. If $\bar{d}_{cc}$ denotes the  Carnot--Carath\'eodory distance with $\{\bar{X}_1,\bar{X}_2\}$ taken as orthonormal basis, it is then clear that
\begin{equation}\label{isometry}
\tilde{d}_{cc} = \bar{d}_{cc}\circ \Phi.
\end{equation}

\begin{lemma}\label{Qj}
Let $\Psi$ be the diffeomorphism defined in \eqref{PsiDiff}. Then the following properties hold.
\begin{itemize} 
\item[(i)] $\psi_j({\boldsymbol \xi}) = x_j$, if $j=1,2$;
\item[(ii)] $\psi_j ({\boldsymbol \xi})= x_j +  \varphi_j(x_1,\dots,x_{\ell})$ for some polynomials $\varphi_j$ and some $\ell<j$, if $j=3,\dots,N$. 
\end{itemize}
\end{lemma}
\begin{proof}
Since $e_N$ is central in $\mathfrak g$, 
$$
\exp\left(\sum_{j=3}^{N-1}x_je_j\right)\exp\left(x_1e_1+x_2e_2\right)\exp\left(x_Ne_N\right) = \exp\left(\sum_{j=3}^{N}x_je_j\right)\exp\left(x_1e_1+x_2e_2\right).
$$
In order to obtain the functions $\psi_j$, $j=1,\dots,N$, we can expand the right hand side of the formula above by means of the Baker--Campbell--Hausdorff formula. For every $j=1,\dots,N$,  $\psi_j$ will be the coefficient of $e_j$, which will be a linear combination of monomials of the form $x_{i_1}\cdot \ldots \cdot x_{i_k}$, whenever the bracket of $k$ vectors $e_{i_1},\dots, e_{i_k}$ taken in any order is a nonzero multiple of $e_j$. Since $\mathfrak g$ is stratified, 
it follows that $\psi_j({\boldsymbol \xi})=x_j$ if $j=1,2$. Furthermore, if $e_j\in{\mathfrak g}_h$, $2\leq h\leq 2k$, then $\psi_j({\boldsymbol \xi})=x_j+\varphi_j(x_1,\dots,x_{\ell})$, where 
$\varphi_j$ is a polynomial in the variables  $(x_1,\dots,x_{\ell})$ so that $e_\ell \in {\mathfrak g}_{h-1}$. This proves (i) and (ii).
\end{proof}

\begin{theorem}\label{mainOt}
The left-invariant vector fields on $(\mathbb R^N, \ast)$ defined by
$$
\tilde{X}_1 = \frac{d}{d\tau}_{\Big|_{t=0}} (x_1,x_2,x_3,\dots,x_{N-1},x_N)\ast (\tau,0,\dots,0)
$$
and
$$
\tilde{X}_2 = \frac{d}{d\tau}_{\Big|_{t=0}} (x_1,x_2,x_3,\dots,x_{N-1},x_N)\ast (0,\tau,\dots,0)
$$
have the form
$$
\tilde{X}_1 =  \partial_{x_1}+ \sum_{j=3}^{N-1} p_j(x_1,x_2)\partial_{x_j}       +x_2(x_1^2+x_2^2)^{k-1} \partial_{x_N},
$$
and 
$$
\tilde{X}_2 =\partial_{x_2}+ \sum_{\ell=3}^{N-1} q_\ell(x_1,x_2)\partial_{x_\ell}  -x_1(x_1^2+x_2^2)^{k-1} \partial_{x_N}
$$
for some polynomials $p_j$, $q_\ell$ and for $j,\ell=3,\dots,N-1$.
\end{theorem}
\begin{proof}
We prove the statement for $\tilde{X}_1$. The proof for $\tilde{X}_2$ is identical.
The flow of $X_1$ is
$$
\Phi_\tau^{X_1} (x_1,x_2,x_N) = \big(x_1+\tau,x_2,x_N+\tau x_2(x_1^2+x_2^2)^{k-1}+O(\tau^2) \big).
$$
The statement is true if we show that there exist polynomials $y_j=y_j({\boldsymbol \xi})$, $j=2,\dots,N-1$ such that
\begin{align*}
 &\exp\left(\sum_{j=3}^{N-1}y_je_j\right) \exp\left((x_1+\tau)e_1+x_2e_2\right)\exp\left((x_N+\tau x_2(x_1^2+x_2^2)^{k-1}+O(\tau^2) )e_N\right)
 \\
 &\qquad=
 \exp\left(\sum_{j=3}^{N-1}x_je_j\right)\exp\left(x_1e_1+x_2e_2\right)\exp\left(x_N e_N\right) \exp(\tau e_1).
\end{align*}
The identity above is equivalent to
\begin{align*}
 &\exp\left(\sum_{j=3}^{N-1}y_je_j\right)\exp\left((\tau x_2(x_1^2+x_2^2)^{k-1}+O(\tau^2) )e_N\right)
\\
&\qquad=
\exp\left(\sum_{j=3}^{N-1}x_je_j\right)\exp\left(x_1e_1+x_2e_2\right)\exp(\tau e_1)\exp\left((-x_1-\tau)e_1-x_2e_2\right).
\end{align*}
We notice that the expansion of $\exp\left(x_1e_1+x_2e_2\right)\exp(\tau e_1)\exp\left((-x_1-\tau)e_1-x_2e_2\right)$ is in $\exp\left( \sum_{j=2}^{2k} \mathfrak{g}_j\right)$. Since $\mathfrak{g}$ is nilpotent, $\sum_{j=2}^{2k-1} \mathfrak{g}_j$ is a sub-algebra and $\mathfrak{g}_{2k}$ is central, all we need to show is the following statement.

\begin{claim}\label{claimOt}  The component along  $e_N$ of $\exp\left(x_1e_1+x_2e_2\right)\exp(\tau e_1)\exp\left((-x_1-\tau)e_1-x_2e_2\right)$ is $\tau x_2(x_1^2+x_2^2)^{k-1}+O(\tau^2)$.
\end{claim}

To this purpose, we extract more information on the brackets that generate $e_N$. 
By the binomial formula, we may write
$$
X_1 = \partial_{x_1} +\sum_{j=0}^{k-1} {k-1 \choose j} x_1^{2j} x_2^{2k-1-2j}\partial_{x_N}
$$
and 
$$
X_2 = \partial_{x_2} -\sum_{j=0}^{k-1} {k-1 \choose j} x_1^{2j+1} x_2^{2k-2-2j}\partial_{x_N}.
$$
Then 
$$
X_{21} := [X_2,X_1] = 2\sum_{j=0}^{k-1} k{k-1 \choose j}  x_1^{2j} x_2^{2k-2-2j} \partial_{x_N}.
$$
Notice that ${\rm ad} X_1\circ {\rm ad} X_2 (X_{21})= \partial_{x_1}\partial_{x_2} (X_{21})=\partial_{x_2}\partial_{x_1} (X_{21})=
{\rm ad} X_2\circ {\rm ad} X_1(X_{21})$.
In order to obtain a non-zero vector in ${\rm span}\{\partial_{x_N}\}$, we must apply ${\rm ad} X_1$ and ${\rm ad} X_2$ to $X_{21}$, $2j$ and $2k-2-2j$ times respectively, in any order, and for each $j=0,\dots,n$. When we do so, we obtain
$$
\left( 2 k{k-1 \choose j} (2j)! (2k-2-2j)! \right) \partial_{x_N},
$$
from which we obtain the formula
\begin{equation}\label{bracketCenter}
{\rm ad}^{2j}e_1\circ {\rm ad}^{2k-2-2j} e_2 ([e_2,e_1])=\left( 2k{k-1 \choose j} (2j)! (2k-2-2j)! \right) e_N.
\end{equation}
Notice that Claim \ref{claimOt} is equivalent to prove that 
\begin{align}\label{finalClaimOt}
&\pi_{\mathfrak{g}_{2k}} \left( \frac{d}{d\tau}{\Big |}_{\tau =0}\exp\left(x_1e_1+x_2e_2\right)\exp(\tau e_1)\exp\left((-x_1-\tau)e_1-x_2e_2\right)\right) \\
&= x_2(x_1^2+x_2^2)^{k-1}e_N,\nonumber
\end{align}
where $\pi_{\mathfrak{g}_{2k}}: \mathfrak{g} \to \mathfrak{g}_{2k}$ is the canonical projection.
We  write 
$$\exp\left(x_1e_1+x_2e_2\right)\exp(\tau e_1)\exp\left((-x_1-\tau)e_1-x_2e_2\right)=\gamma_1(\tau) \gamma_2(\tau),$$ 
with 
$$\gamma_1(\tau) = \exp\left(x_1e_1+x_2e_2\right)\exp(\tau e_1)\exp\left(-x_1e_1-x_2e_2\right)$$
 and 
 $$\gamma_2(\tau) =\exp\left(x_1e_1+x_2e_2\right)\exp\left((-x_1-\tau)e_1-x_2e_2\right).$$
Notice that $\gamma_1(0) =\gamma_2(0) =e$.
Using the product rule and formula $(2.14.2)$ in \cite{VaradarajanLieGroups}, and the fact that $\mathfrak{g}$ is nilpotent of step $2k$, we have that
\begin{align}\label{generalLift}
\frac{d}{d\tau}{\Big |}_{\tau=0}&\exp\left(x_1e_1+x_2e_2\right)\exp(\tau e_1)\exp\left((-x_1-\tau)e_1-x_2e_2\right)\nonumber\\
&= \frac{d}{d\tau}{\Big |}_{\tau=0}\gamma_1(\tau)\gamma_2(\tau)\nonumber \\
&=e^{{\rm ad}(x_1e_1+x_2e_2)}(e_1)-\sum_{j=0}^{2k-1} \frac{(-1)^j}{(j+1)!}{\rm ad}^j(-x_1e_1-x_2e_2)(e_1)\nonumber\\
&= \sum_{\ell =0}^{2k-1}  \frac{1}{\ell!}{\rm ad}^\ell(x_1e_1+x_2e_2)(e_1)-\sum_{j=0}^{2k-1} \frac{1}{(j+1)!}{\rm ad}^j(x_1e_1+x_2e_2)(e_1)\nonumber\\
&= \sum_{\ell =0}^{2k-1} \left( \frac{1}{\ell!}-  \frac{1}{(\ell+1)!}\right){\rm ad}^\ell(x_1e_1+x_2e_2)(e_1).
\end{align}
Applying $\pi_{\mathfrak{g}_{2k}}$, we obtain that 
\begin{align*}
&\pi_{\mathfrak{g}_{2k}}\left(\frac{d}{d\tau}{\Big |}_{\tau=0}\exp\left(x_1e_1+x_2e_2\right)\exp(\tau e_1)\exp\left((-x_1-\tau)e_1-x_2e_2\right)\right) \\
&= \frac{{\rm ad}^{2k-1}(x_1e_1+x_2e_2)(e_1)}{2k(2k-2)!}.
\end{align*}
Next, \eqref{bracketCenter} implies that
\begin{align*}
{{\rm ad}^{2k-1}(x_1e_1+x_2e_2)(e_1)}&=x_2{{\rm ad}^{2k-2}(x_1e_1+x_2e_2)(e_{21})}\\
&=x_2 \sum_{j=0}^{2k-2}{2k-2 \choose j} {\rm ad}^j {(x_1e_1)}\circ {\rm ad}^{2k-2-j}(x_2e_2)(e_{21})\\
&=x_2 \sum_{j=0}^{k-1}{2k-2 \choose 2j} {\rm ad}^{2j} {(x_1e_1)}\circ {\rm ad}^{2k-2-2j}(x_2e_2)\\
&= x_2\sum_{j=0}^{k-1}2k{2k-2 \choose 2j} {k-1\choose j}(2j)!(2k-2-2j)x_1^{2j}x_2^{2k-2-2j}e_N\\
&=2k(2k-2)! x_2 \sum_{j=0}^{k-1} {k-1\choose j} x_1^{2j}x_2^{2k-2-2j}e_N\\
&= 2k(2k-2)! x_2 (x_1^2+x_2^2)^{k-1}e_N.
\end{align*}
Hence, 
$$
\frac{{\rm ad}^{2k-1}(x_1e_1+x_2e_2)(e_1)}{2k(2k-2)!} = x_2 (x_1^2+x_2^2)^{k-1}e_N
$$
and \eqref{finalClaimOt} follows, concluding the proof.
\end{proof}

\begin{example}
We consider the case $k=2$ and show the construction of the Lie algebra $\mathfrak g$ and the lift vector fields $\tilde{X}_j$, $j=1,2$. We start with the vector fields in $\mathbb R^3$ given by
$$ X_1= \partial_{x_1}+x_2(x_1^2+x_2^2) \partial_{t}, $$
$$ X_2= \partial_{x_2} -x_1(x_1^2+x_2^2) \partial_{t}. $$
The complex vector field $L=X_1+iX_2$ defines a CR-structure on the manifold
$$
\partial\Omega_2= \Big\{(z,w)\in {\mathbb C}^2: {\rm Im}w = \frac{1}{4}(x_1^2+x_2^2)^{2}\Big\},
$$
where $z=x_1+ix_2$ and ${\rm Re}(w)= t$.
The vector fields $X_1$ and $X_2$, for $x_1^2+x_2^2\neq 0$, form a $6$-dimensional  Lie algebra 
$\mathfrak{g}={\rm span}\{e_1,\dots,e_6\}$ where the nontrivial brackets are given by
$$
e_3=[e_2,e_1], \quad e_4=[e_3,e_1], \quad e_5=[e_3,e_2],\quad [e_4,e_1]=[e_5,e_2]=8e_6.
$$
Denoted by $G$ the connected and simply connected Lie group with Lie algebra $\mathfrak{g}$, it is easy to compute the left-invariant vector fields on G corresponding to $e_1$ and $e_2$. Using the coordinates defined by $\Phi$ and \eqref{generalLift}, we obtain
$$
\tilde{X}_1 =  \partial_{x_1}+ \frac{x_2}{2}\partial_{x_3}-\frac{1}{3}x_1x_2 \partial_{x_4}-\frac{1}{3}x^2_2 \partial_{x_5}+x_2(x_1^2+x_2^2) \partial_{t}
$$
and 
$$
\tilde{X}_2 =  \partial_{x_2}- \frac{x_1}{2}\partial_{x_3}+\frac{1}{3}x^2_2 \partial_{x_4}+\frac{1}{3}x_1x_2 \partial_{x_5}-x_1(x_1^2+x_2^2) \partial_{t}.
$$
\end{example}

\begin{remark}\label{constantOnCosets}
 $(i)$ Notice that for $j=1,2$ we have that $\tilde{X}_j = X_j$ on functions constant  in the  variables $x_3,\dots,x_{N-1}$.
 
$(ii)$ Once we fix a basis of the Lie algebra $\mathfrak{g}$, we can use the formula \eqref{generalLift} in the proof of Theorem~\ref{mainOt} to compute explicitly $\tilde{X}_1$ and $\tilde{X}_2$ like we did in the example above.
\end{remark}

\subsection{First order Taylor's expansion}
We will need to approximate functions on $\partial\Omega_k$ by their first order Taylor expansion. In order to do that, we first obtain a first order approximation with an estimate of the remainder for the group $G$ introduced earlier. This is an easy consequence of \cite[Corollary 1.44]{FS}.
We remind that 
polynomials on nilpotent groups are well defined objects, in the sense that if a function
on G is polynomial in exponential coordinates of the first type, then it is also polynomial in every system of exponential coordinates of the second type.

\begin{lemma}\label{TaylorForG}
There exist positive constants $b,C$ such that for all twice differentiable functions $F$  on $(\mathbb R^N, \ast)$,
$$
| F({\boldsymbol \xi}{\boldsymbol \eta}) - P_{\boldsymbol \xi}({\boldsymbol \eta}) | \leq C \tilde{d}_{cc}({\boldsymbol \eta},{\bf 0})^2 \sup_{\tilde{d}_{cc}({\boldsymbol \zeta},{\bf 0})\leq b\tilde{d}_{cc}({\boldsymbol \eta},{\bf 0}), i,j=1,2} | \tilde{X}_i \tilde{X}_j F({\boldsymbol \xi}{\boldsymbol \zeta}) |,
$$
for all ${\boldsymbol \xi},{\boldsymbol \eta}\in \mathbb R^N$,
where $P_{\boldsymbol \xi}({\boldsymbol \eta}) = F({\boldsymbol \xi})+ y_1 \tilde{X}_1 F({\boldsymbol \xi}) + y_2 \tilde{X}_2 F({\boldsymbol \xi})$.
\end{lemma}
\begin{proof}
From \cite[Corollary 1.44]{FS},  then
there exist $b,C\in \mathbb R$ such that
\begin{equation}\label{TaylorForFirstType}
| F(pq) - P_p(q) | \leq C' \bar{d}_{cc}(q,0)^2 \sup_{\bar{d}_{cc}(r,0)\leq b'\bar{d}_{cc}(q,0), i,j=1,2} | \bar{X}_i \bar{X}_j F(pr) |,
\end{equation}
for all $p=\exp(\sum_{j=1}^N u_i e_i)$ and $q=\exp(\sum_{j=1}^N v_i e_i)$.
Here
$P_p(q) = F(p)+ v_1 \bar{X}_1 F(p) + v_2 \bar{X}_2 F(p)$.
 From Lemma~\ref{Qj}, if
 $$\Phi({\boldsymbol \xi}) = \exp\left(\sum_{j=3}^{N-1}x_je_j\right)\exp\left(x_1e_1+x_2e_2\right)\exp\left(x_N e_N\right),$$ 
then $\Phi({\boldsymbol \xi}) = \exp\left(x_1e_1+x_2e_2+ \sum_{j=3}^N (x_j+ \varphi_j(x_1,\dots,x_\ell))e_j\right)$ for some polynomials $\varphi_j$, $j=3,\dots,N$, and some $\ell<j$. 
Apply $\eqref{TaylorForFirstType}$ to the points 
$$p= \exp\left(x_1e_1+x_2e_2+ \sum_{j=3}^N \varphi_j({\boldsymbol \xi})e_j\right),$$
 $$q= \exp\left(y_1e_1+y_2e_2+ \sum_{j=3}^N \varphi_j({\boldsymbol \eta})e_j\right),$$
 and
 $$r= \exp\left(z_1e_1+z_2e_2+ \sum_{j=3}^N \varphi_j({\boldsymbol \zeta})e_j\right).$$
Notice that $\bar{X}_jF(\Phi(\cdot))=\tilde{X}_jF(\cdot)$, $j=1,2$.
This yields 
$F(pq) - P_p(q) = F({{\boldsymbol \xi}{\boldsymbol \eta}}) -P_{\boldsymbol \xi}({\boldsymbol \eta})$ and $ \bar{X}_i \bar{X}_j F(pr)=\tilde{X}_i \tilde{X}_j F({\boldsymbol \xi}{\boldsymbol \zeta})$. 
Moreover, \eqref{isometry} implies that $\bar{d}_{cc}(\Phi({\cdot}),e)=\tilde{d}_{cc}(\cdot,{\bf 0})$, which concludes the proof.
\end{proof}

\section{Taylor expansion on $\partial \Omega_k$}\label{Sec: Tay}
In this section, we will adapt Lemma~\ref{TaylorForG} to $\partial \Omega_k$. While we can obtain estimates for the remainder of the Taylor polynomial of any order, in this paper we will only need the remainder of order 2. For notational convenience, we will then only describe the Taylor polynomial of order $1$ for functions on $\partial \Omega_k$ with remainder of order $2$.
In order to adapt Lemma~\ref{TaylorForG} to $\partial \Omega_k$, we  need to establish a relation between the Carnot--Carath\'eodory distance on 
$\partial \Omega_k$ and that on $G$, by following ideas of \cite{Nagel-Stein-Wainger} and \cite{San}.

Define 
\begin{align*}
d_{cc}( {\bf x}, {\bf y}) &= \inf\{\tau: \exists\ {\rm horizontal \ curve\ }\ \gamma: [0,\tau]\to\partial\Omega_k \ {\rm s.t.\ } \ \|\dot\gamma(s)\|_{cc}\leq1\\
 &\hskip2cm  {\rm\ for\ } s\in[0,\tau], \gamma(0)={\bf x}, \gamma(\tau)={\bf y} \},
\end{align*}
where
$$ \|c_1X_1+c_2X_2\|_{cc}:= \sqrt{c_1^2+c_2^2}.  $$
We denote by $B_{d_{cc}}$ and $B_{\tilde{d}_{cc}}$ the balls for the metrics $d_{cc}$ and $\tilde{d}_{cc}$, respectively.
We define the embedding $\Theta: \partial\Omega_k\to (\mathbb R^N,\ast)$ as follows. 
For every ${\bf x}=(\xi_1,\xi_2,\xi_N)\in \partial\Omega_k$, let $\Theta({\bf x}) = (\xi_1,\xi_2,0,\dots,0,\xi_N)$. The inverse $\Theta^{-1}$  is the canonical projection restricted to $\Theta(\partial\Omega_k)$. 
\begin{proposition}\label{metric equivalence}
For every ${\bf x},{\bf y}\in\partial\Omega_k$, we have $\tilde{d}_{cc}(\Theta({\bf x}),\Theta({\bf y}))= d_{cc}({\bf x},{\bf y})$.
\end{proposition}
\begin{proof}

Let ${\bf x},{\bf y}\in\partial\Omega_k$.
Since every horizontal curve on $\partial\Omega_k$ is horizontal on $\mathbb R^N$, it follows that $\tilde{d}_{cc}(\Theta({\bf x}),\Theta({\bf y}))\leq d_{cc}({\bf x},{\bf y})$.

Next, let $\tilde \gamma: [0,\tau]\to \mathbb R^N$ be a horizontal curve in $\mathbb R^N$ with $\|\tilde\gamma(s)\|_{\tilde{cc}}\leq1$, $\tilde \gamma(0)= \Theta({\bf x})$ and  $\tilde \gamma(\tau)= \Theta({\bf y})$. In particular, 
$$ \dot{ \tilde{\gamma}}(s) =\alpha_1(s)\tilde X_1+\alpha_2(s)\tilde X_2 $$
with
$$ \alpha_1(s)^2+\alpha_2(s)^2\leq1, \quad \forall s\in[0,\tau]. $$
Using~Theorem \ref{mainOt}, it follows that the components of $\dot{ \tilde{\gamma}}$ along $e_1,e_2$, and $e_N$ are
\begin{equation*}
\left\{ \begin{aligned} 
  \dot{ \tilde{\gamma}}_1(s) &= \alpha_1(s),\\
  \dot{ \tilde{\gamma}}_2(s) &= \alpha_2(s),\\
    \dot{ \tilde{\gamma}}_N(s) &= \alpha_1(s)\tilde{\gamma}_2(s)\Big[ (\tilde{\gamma}_1(s))^2+(\tilde{\gamma}_2(s))^2 \Big]^{k-1}-\alpha_2(s)\tilde{\gamma}_1(s)\Big[ (\tilde{\gamma}_1(s))^2+(\tilde{\gamma}_2(s))^2 \Big]^{k-1}.
\end{aligned} \right.
\end{equation*}
The projection of ${ \tilde{\gamma}}$ onto $\partial\Omega_k$ is the curve $\gamma :[0,\tau]\to\partial\Omega_k$  such that 
$$\gamma(s)=( \tilde{\gamma}_1(s),\tilde{\gamma}_2(s),\tilde{\gamma}_N(s) )\in\partial\Omega_k$$ and
\begin{equation*}
\left\{ \begin{aligned} 
  \dot{ \tilde{\gamma}}_1(s) &= \alpha_1(s),\\
  \dot{ \tilde{\gamma}}_2(s) &= \alpha_2(s),\\
    \dot{ \tilde{\gamma}}_N(s) &= \alpha_1(s)\tilde{\gamma}_2(s)\Big[ (\tilde{\gamma}_1(s))^2+(\tilde{\gamma}_2(s))^2 \Big]^{k-1}-\alpha_2(s)\tilde{\gamma}_1(s)\Big[ (\tilde{\gamma}_1(s))^2+(\tilde{\gamma}_2(s))^2 \Big]^{k-1}.
\end{aligned} \right.
\end{equation*}
Then it is direct that $\gamma(0)={\bf x}$ and $\gamma(\tau)={\bf y}$.
Moreover, we have
\begin{align*}
  \dot{\gamma}(s) &=\dot{ \tilde{\gamma}}_1(s){\partial_{x_1}}+\dot{ \tilde{\gamma}}_2(s){\partial_{x_2}}+\dot{ \tilde{\gamma}}_N(s){\partial_{x_N}}\\
  &=\alpha_1(s){\partial_{x_1}}+\alpha_2(s){\partial_{x_2}}\\
  &\quad+\Bigg[\alpha_1(s)\tilde{\gamma}_2(s)\Big[ (\tilde{\gamma}_1(s))^2+(\tilde{\gamma}_2(s))^2 \Big]^{k-1}-\alpha_2(s)\tilde{\gamma}_1(s)\Big[ (\tilde{\gamma}_1(s))^2+(\tilde{\gamma}_2(s))^2 \Big]^{k-1}\Bigg]{\partial_{x_N}}\\
  &= \alpha_1(s)X_1+\alpha_2(s)X_2,
 \end{align*}
where the last equality follows from the definition of $X_1$ and $X_2$.

Thus, $\gamma$ is a horizontal curve on $\partial\Omega_k$ with $\|\gamma\|_{cc}\leq1$ and $\tau\geq d_{cc}({\bf x},{\bf y})$. By definition, we conclude that
$\tilde{d}_{cc}(\Theta({\bf x}),\Theta({\bf y}))\geq d_{cc}({\bf x},{\bf y})$.
\end{proof}

We can now prove the following result for the first order Taylor polynomial on $\partial\Omega_k$.
\begin{theorem}\label{TaylorForOmega}
There exist positive constants $b,C$ such that for all twice differentiable functions $f$  on $\partial\Omega_k$,
$$
| f({\bf y}) - P_{\bf x}({{\bf y}-{\bf x}}) | \leq C d_{cc}({\bf x},{{\bf y}})^2 \sup_{d_{cc}({\bf x},{{\bf z}})\leq b\,d_{cc}({\bf x},{{\bf y}}), i,j=1,2} | {X}_i {X}_j f({{\bf z}}) |,
$$
for all ${\bf x},{{\bf y}}\in \partial\Omega_k$, and with
$P_{\bf x}({\bf y}-{\bf x}) = f({\bf x})+(y_1-{x}_1)X_1f({\bf x})+(y_2-{x}_2)X_2f({\bf x})$.
\end{theorem}
\begin{proof}

Apply Lemma~\ref{TaylorForG} with ${\boldsymbol \eta}={\boldsymbol \xi}^{-1}\ast {\boldsymbol \eta}$ and ${\boldsymbol \zeta}={\boldsymbol \xi}^{-1}\ast {\boldsymbol \zeta}$. We obtain 
\begin{equation}\label{TaylorTranslated}
| F({\boldsymbol \eta}) - P_{\boldsymbol \xi}({\boldsymbol \xi}^{-1}\ast {\boldsymbol \eta}) | \leq C \tilde{d}_{cc}({\boldsymbol \xi},{\boldsymbol \eta})^2 \sup_{\tilde{d}_{cc}({\boldsymbol \xi},{\boldsymbol \zeta})\leq b\tilde{d}_{cc}({\boldsymbol \xi},{\boldsymbol \eta}), i,j=1,2} | \tilde{X}_i \tilde{X}_j F({\boldsymbol \zeta}) |,
\end{equation}
Embed $\partial\Omega_k$ into $(\mathbb R^N, \ast)$ by $\Theta(p_1,p_2,p_N)=(p_1,p_2,0,\dots,0,p_N)$ for every  $(p_1,p_2,p_N) \in \partial\Omega_k$. For every  twice differentiable function $f$  on $\partial\Omega_k$, define $F$ on $(\mathbb R^N, \ast)$ by $$F(x_1,x_2,x_3,\dots,x_{N-1},x_N)=f(\Theta^{-1}(x_1,x_2,0,\dots,0,x_N)),$$
with $\Theta^{-1}$ well defined on the image of $\Theta$.
Hence, apply \eqref{TaylorTranslated} to this $F$, with  ${\boldsymbol \xi}=\Theta({\bf x})$, ${\boldsymbol \eta}=\Theta({\bf y})$, with ${\bf x} = (x_1,x_2,x_N)$ and ${\bf y} = (y_1,y_2,y_N)$. Recalling from Proposition~\ref{metric equivalence} that $\tilde{d}_{cc} \circ \Theta= d_{cc}$, we obtain that
$$
|f({\bf y})-P_{\bf x}({\bf y}-{\bf x})|\leq C{d}_{cc}({\bf x},{\bf y})^2\sup_{\tilde{d}_{cc}(\Theta({\bf x}),{\boldsymbol \zeta})\leq b{d}_{cc}({\bf x},{\bf y}), i,j=1,2}  | \tilde{X}_i \tilde{X}_j f(\Theta^{-1}({z_1},z_2,0,\dots,0,z_N)) |,
$$
where $P_{\bf x}({\bf y}-{\bf x}) = f({\bf x})+(y_1-{x}_1)X_1f({\bf x})+(y_2-{x}_2)X_2f({\bf x})$. Here we use (i) of Lemma~\ref{Qj} to show that the first two components of ${\boldsymbol \xi}^{-1}\ast {\boldsymbol \eta}$ are just the linear differences of the first two components of $\bf y$ and $\bf x$, respectively. Further, (i) of Remark~\ref{constantOnCosets} implies that $\tilde{X}_jF\circ\Theta = X_j f$.
Let ${\boldsymbol \zeta}=(z_1,z_2,z_3,\dots,z_{N-1},z_N)$. Since $f\circ \Theta^{-1}$ depends only on ${\boldsymbol \zeta}_0=(z_1,z_2,0,\dots,0,z_N)= \Theta({\bf z})$ with ${\bf z} = (z_1,z_2,z_N)$, we see that 
\begin{align*}
\sup_{\tilde{d}_{cc}(\Theta({\bf x}),{\boldsymbol \zeta})\leq b{d}_{cc}({\bf x},{\bf y}), i,j=1,2}  | \tilde{X}_i \tilde{X}_j f(\Theta^{-1}({\boldsymbol \zeta}_0) | &=\sup_{\tilde{d}_{cc}(\Theta({\bf x}),{\boldsymbol \zeta}_0)\leq b{d}_{cc}({\bf x},{\bf y}), i,j=1,2}  | \tilde{X}_i \tilde{X}_j f(\Theta^{-1}({\boldsymbol \zeta}_0) |\\
&\leq \sup_{{d}_{cc}({\bf x},{\bf z})\leq b{d}_{cc}({\bf x},{\bf y}), i,j=1,2}  | {X}_i {X}_j f({\bf z}) |.
\end{align*}
 We conclude that 
$$
|f({\bf y})-P_{\bf x}({\bf y}-{\bf x})|\leq C{d}_{cc}({\bf x},{\bf y})^2\sup_{{d}_{cc}({\bf x},{\bf z})\leq b{d}_{cc}({\bf x},{\bf y}), i,j=1,2}  | {X}_i {X}_j f({\bf z}) |
$$
as required.
\end{proof}

\section{\bf Commutator of Cauchy--Szeg\"o projection: Proof of (1) of Theorem \ref{main thm}} 

In this section we  prove $(1)$ of Theorem \ref{main thm}. Our  approach is based first on a fundamental non-degeneracy property of the Cauchy--Szeg\"o kernel showed in the first lemma below, and secondly on  
 techniques in dyadic harmonic analysis that allow us to bypass the use of Cayley and Fourier transforms or the group structure used in \cite{RochSemm,FR,FLL}.  
\begin{lemma}\label{lemA4}
For each dyadic cube $Q$, there exists another dyadic cube $\hat{Q}$ such that

{\rm{(i)}} $|Q|=|\hat{Q}|$, and $\mathrm{distance}(Q,\hat{Q})\approx |Q|$.

{\rm{(ii)}} $S_1(\mathbf x, \hat{\mathbf x})$ $($or $S_2(\mathbf x, \hat{\mathbf x})$$)$ does not change sign for all $(\mathbf x, \hat{\mathbf x}) \in Q \times \hat{Q}$ and
\begin{align}\label{Kernel}
|S_1(\mathbf x, \hat{\mathbf x})| \gtrsim \frac{1}{|Q|}\quad \Big(\text{or}~~|S_2(\mathbf x, \hat{\mathbf x})| \gtrsim \frac{1}{|Q|}\Big) ,
\end{align}
where $S_1(\mathbf x, \hat{\mathbf x})$ and $S_2(\mathbf x, \hat{\mathbf x})$ are the real and imaginary part of the Cauchy--Szeg\"o kernel $S(\mathbf x, \hat{\mathbf x})$, respectively.
\end{lemma}

\begin{proof}
Assume that $Q$ is any dyadic cube of generation $j$ with center point $\mathbf x_0\in Q$ and side length $2^{-j}$. By definition,
$$B(\mathbf x_0, a_12^{-j})\subseteq Q\subseteq B(\mathbf x_0, A_12^{-j}):=B.$$
Let $\hat B=B(\hat {\mathbf x}_0, A_12^{-j})$ satisfy 
\begin{align}\label{a34}
A_3C_d 2^{-j} <d(\mathbf x_0, \hat{\mathbf x}_0)<A_4C_d 2^{-j},
\end{align}
where 
$A_3>\max\big\{2, {1\over C_d}(8\sqrt 2 C_1 C_d^{2})^{2k+2}\big\}$ and $A_4$ is large enough such that there exists another dyadic cube in $\mathscr D_j$ with $\hat Q\subseteq \hat B$, then $|Q|\approx|\hat{Q}|.$

From \eqref{k-size1} we can see $$|S (\mathbf x_0, \hat{\mathbf x}_0)|={1\over  d (\mathbf x_0, \hat{\mathbf x}_0)},$$  therefore,
$${\rm either\ \ } |S_1 (\mathbf x_0, \hat{\mathbf x}_0)|\geq {1\over  \sqrt 2 d (\mathbf x_0, \hat{\mathbf x}_0)}\quad\text{or}\quad
|S_2 (\mathbf x_0, \hat{\mathbf x}_0)|\geq {1\over  \sqrt 2 d (\mathbf x_0, \hat{\mathbf x}_0)}.$$
We may assume that the first inequality holds and $S_1 (\mathbf x_0, \hat{\mathbf x}_0)>0$. Then for any
$(\mathbf x, \hat{\mathbf x}) \in B\times\hat B$, by \eqref{k-size2}, we have
\begin{align*}
S_1 (\mathbf x, \hat{\mathbf x})&=S_1 (\mathbf x_0, \hat{\mathbf x}_0)-\left(S_1 (\mathbf x_0, \hat{\mathbf x}_0)-S_1 (\mathbf x, \hat{\mathbf x}) \right)\\
&\geq S_1 (\mathbf x_0, \hat{\mathbf x}_0)-\left|S_1 (\mathbf x_0, \hat{\mathbf x}_0)-S_1 (\mathbf x, \hat{\mathbf x}) \right|\\
&\geq S_1 (\mathbf x_0, \hat{\mathbf x}_0)-\left|S_1 (\mathbf x_0, \hat{\mathbf x}_0)-S_1 (\mathbf x_0, \hat{\mathbf x}) \right|-\left|S_1 (\mathbf x_0, \hat{\mathbf x})-S_1 (\mathbf x, \hat{\mathbf x}) \right|\\
&\geq {1\over  \sqrt 2 d (\mathbf x_0, \hat{\mathbf x}_0)}
- {C_1\over  d (\mathbf x_0, \hat{\mathbf x}_0)}\left({d(\hat{\mathbf x}_0,\hat{\mathbf x})\over d({\mathbf x}_0,\hat{\mathbf x}_0)} \right)^{1\over 2k+2}
-{C_1\over  d (\mathbf x_0, \hat{\mathbf x})}\left({d({\mathbf x}_0,{\mathbf x})\over d({\mathbf x}_0,\hat{\mathbf x})} \right)^{1\over 2k+2}.
\end{align*}

By \eqref{cd} and \eqref{a34}, 
\begin{align*}
d (\mathbf x_0, \hat{\mathbf x})
\geq {1\over C_d}d (\mathbf x_0, \hat{\mathbf x}_0)-d(\hat{\mathbf x}_0,\hat{\mathbf x})
\geq \left({1\over C_d}-{1\over A_3 C_d}\right)d (\mathbf x_0, \hat{\mathbf x}_0).
\end{align*}
Therefore,
\begin{align*}
S_1 (\mathbf x, \hat{\mathbf x})
&\geq {1\over  \sqrt 2 d (\mathbf x_0, \hat{\mathbf x}_0)}
- {C_1\over  d (\mathbf x_0, \hat{\mathbf x}_0)}\left({d(\hat{\mathbf x}_0,\hat{\mathbf x})\over d({\mathbf x}_0,\hat{\mathbf x}_0)} \right)^{1\over 2k+2}
-{C_1\over  d (\mathbf x_0, \hat{\mathbf x})}\left({d({\mathbf x}_0,{\mathbf x})\over d({\mathbf x}_0,\hat{\mathbf x})} \right)^{1\over 2k+2}\\
&\geq {1\over  \sqrt 2 d (\mathbf x_0, \hat{\mathbf x}_0)}
- {C_1\over  d (\mathbf x_0, \hat{\mathbf x}_0)}\left({d(\hat{\mathbf x}_0,\hat{\mathbf x})\over d({\mathbf x}_0,\hat{\mathbf x}_0)} \right)^{1\over 2k+2}\\
&\quad
-{1\over [(1-{1\over A_3}){1\over C_d} ]^{1+{1\over 2k+2}}}{C_1\over  d (\mathbf x_0, \hat{\mathbf x}_0)}\left({d({\mathbf x}_0,{\mathbf x})\over d({\mathbf x}_0,\hat{\mathbf x}_0)} \right)^{1\over 2k+2}\\
&\geq {1\over  \sqrt 2 d (\mathbf x_0, \hat{\mathbf x}_0)}
-\left({1\over A_3C_d} \right)^{1\over 2k+2}{C_1\over  d (\mathbf x_0, \hat{\mathbf x}_0)}\\
&\quad
-{1\over [(1-{1\over A_3}){1\over C_d} ]^{1+{1\over 2k+2}}}\left({1\over A_3C_d} \right)^{1\over 2k+2}{C_1\over  d (\mathbf x_0, \hat{\mathbf x}_0)}\\
&=\left\{{1\over  \sqrt 2}-C_1 \left({1\over A_3C_d} \right)^{1\over 2k+2}-{C_1\over [(1-{1\over A_3}){1\over C_d} ]^{1+{1\over 2k+2}}}\left({1\over A_3C_d} \right)^{1\over 2k+2}\right\}
{1\over d (\mathbf x_0, \hat{\mathbf x}_0)}\\
&\geq \left\{{1\over  \sqrt 2}-4C_d^2C_1 \left({1\over A_3C_d} \right)^{1\over 2k+2}\right\}
{1\over d (\mathbf x_0, \hat{\mathbf x}_0)}\\
&\geq {\sqrt 2\over 4 } {1\over d (\mathbf x_0, \hat{\mathbf x}_0)}.
\end{align*}
Consequently, for any $(\mathbf x, \hat{\mathbf x}) \in Q \times \hat{Q}$, $S_1 (\mathbf x, \hat{\mathbf x})>0$ and 
$$S_1 (\mathbf x, \hat{\mathbf x})\gtrsim {1\over d (\mathbf x_0, \hat{\mathbf x}_0)} \gtrsim {1\over |Q|}. $$
This finishes the proof of Lemma \ref {lemA4}.
\end{proof}

\subsection{Besov space $B_{p}(\partial\Omega_k)$ and its dyadic structure}\label{sec:dy Bp}

We start by defining the norm of the dyadic Besov space.

\begin{definition}
Suppose $0<p<\infty$. Let $b\in L^{1}_{loc}(\partial\Omega_k,\mu)$ and $\mathscr{D}$ be an arbitrary dyadic system in $\partial\Omega_k$. Then $b$ belongs to the dyadic Besov space $B^{d}_{p}(\partial\Omega_k,\mathscr{D})$ if
\begin{align*}
\|b\|_{B^{d}_{p}(\partial\Omega_k,\mathscr{D})}&:=\bigg(\sum_{Q\in\mathscr{D},\epsilon\not\equiv1}\left(|{\langle b, h^{\epsilon}_{Q}\rangle}||Q|^{-\frac{1}{2}} \right)^p\bigg)^{\frac{1}{p}}<\infty.
\end{align*}
\end{definition}

Key to the analysis will be the fact that a suitable family of dyadic norms is equivalent to the norm in the continuous setting, which is the content of the next lemma.

\begin{lemma}\label{lemA1}
Suppose $1<p<\infty$. There are dyadic systems $\mathscr{D}^{\omega}, \omega\in \{1, 2, \ldots, \mathpzc T\}$, such that $\bigcap_{\omega=1}^{\mathpzc T} B^{d}_{p}(\partial\Omega_k,\mathscr{D}^{\omega})= B_{p}(\partial\Omega_k)$
with
$\sum_{\omega=1}^{\mathpzc T}\|b\|_{B^{d}_{p}(\partial\Omega_k,\mathscr{D}^{\omega})}\approx \|b\|_{B_{p}(\partial\Omega_k)}.$
\end{lemma}

\color{black}

\begin{proof}
On the one hand, we  prove that the dyadic Besov space norm $\|b\|_{B^{d}_{p}(\partial\Omega_k,\mathscr{D}^{\omega})}$ is dominated by the continuous Besov space norm for $\omega=1,\ldots, \mathpzc T$. For notational convenience, we drop the superscript $\omega$, that is, we aim to prove that 
\begin{align}\label{dy Besov to Besov}
\|b\|_{B^{d}_{p}(\partial\Omega_k,\mathscr{D})}
\lesssim \|b\|_{B_{p}(\partial\Omega_k)}.
\end{align}

For any fixed dyadic cube $Q\in \mathscr{D}$ of generation $j$ with center point $\mathbf x^j$, let $R_Q=B(\mathbf x^j, 3A_2^2 2^{-j})\setminus B(\mathbf x^j, 2A_2^2 2^{-j})$, where $A_2=\max\{ A_1, C_d\}$, here $ A_1$ and $C_d$ are the constants in the definition of a system of dyadic cubes and \eqref{cd}, respectively. Then if $\mathbf x\in Q$, $\mathbf y\in R_Q$, by \eqref{cd}, we claim that 
\begin{equation}\label{est-dxy}
A_22^{-j}<d(\mathbf x, \mathbf y)<2A_2^3 2^{-j+1}.
\end{equation}

In fact, we have the following estimate:
 \begin{align*}
 d(\mathbf x, \mathbf y)&\leq A_2\left(d(\mathbf x, \mathbf x^{j}) +d(\mathbf y, \mathbf x^{j}) \right)
 <A_2\left(A_2 2^{-j}+3A_2^2 2^{-j} \right)<4A_2^3 2^{-j},\\
  d(\mathbf x, \mathbf y)&\geq {1\over A_2}d(\mathbf y, \mathbf x^{j}) -d(\mathbf x, \mathbf x^{j})
  > {1\over A_2} 2A_2^2 2^{-j}-A_2 2^{-j}=A_22^{-j}.
\end{align*}

By applying H\"{o}lder's inequality 
\begin{align*}
|\langle b, h^{\epsilon}_{Q}\rangle||Q|^{-\frac{1}{2}}&\lesssim \bigg|\int_{Q}b(\mathbf x)h^{\epsilon}_{Q}(\mathbf x)d\mathbf x\bigg|\frac{|R_Q|}{|Q|^{\frac{3}{2}}}
\lesssim \int_{R_Q}\int_{Q}|b(\mathbf x)-b(\mathbf y)||h^{\epsilon}_{Q}(\mathbf x)|d\mathbf xd\mathbf y|Q|^{-\frac{3}{2}}
\\&\lesssim \bigg(\int_{R_Q}\int_{Q}\frac{|b(\mathbf x)-b(\mathbf y)|^{p}}{d(\mathbf x, \mathbf y)^2}d\mathbf xd\mathbf y\bigg)^{\frac{1}{p}}
\bigg(\int_{R_Q}\int_{Q}|h^{\epsilon}_{Q}(\mathbf x)|^{p'}d\mathbf xd\mathbf y\bigg)^{\frac{1}{p'}}{|Q|^{-\frac{2}{p'}+\frac{1}{2}}}\\
&\lesssim \bigg(\int_{R_Q}\int_{Q}\frac{|b(\mathbf x)-b(\mathbf y)|^{p}}{d(\mathbf x, \mathbf y)^2}d\mathbf xd\mathbf y\bigg)^{\frac{1}{p}}.
\end{align*}
Hence, by \eqref{est-dxy}, we can obtain
 \begin{align*}
\|b\|^{p}_{B^{d}_{p}(\partial\Omega_k,\mathscr{D})}&=\sum_{{\substack{Q\in\mathscr{D}\\\epsilon\not\equiv1}}}\left(|{\langle b, h^{\epsilon}_{Q}\rangle}||Q|^{-\frac{1}{2}} \right)^p \lesssim \sum_{{Q\in\mathscr{D}}}\int_{R_Q}\int_{Q}\frac{|b(\mathbf x)-b(\mathbf y)|^{p}}{d(\mathbf x, \mathbf y)^2}d\mathbf xd\mathbf y 
\\
&{\lesssim \sum_{j\in \mathbb{Z}}\sum_{{Q\in\mathscr{D}_j}}\int_{Q}
\int_{\{y\in\partial\Omega_k: A_22^{-j}<d(\mathbf x, \mathbf y)<2A_2^3 2^{-j+1}\}}\frac{|b(\mathbf x)-b(\mathbf y)|^{p}}{d(\mathbf x, \mathbf y)^2}d\mathbf yd\mathbf x 
}
\lesssim \|b\|^{p}_{B_{p}(\partial\Omega_k)}.
\end{align*}
This implies that \eqref{dy Besov to Besov} holds.

\medskip

On the other hand, we need to consider that
\begin{align}\label{eq-S2.4}
\|b\|_{B_{p}(\partial\Omega_k)}\lesssim \sum_{\omega=1}^{\mathpzc T}\|b\|_{B^{d}_{p}(\partial\Omega_k,
\mathscr{D}^{\omega})}.
\end{align}

We first claim that 
\begin{equation}\label{decom}
\partial\Omega_k\times \partial\Omega_k=\bigcup_{j\in\mathbb Z}\bigcup_{Q\subset \mathscr{D}_j}\Gamma_Q,
\end{equation}
where $\Gamma_Q=\{(\mathbf x, \mathbf y): \mathbf x\in Q, 2^{-j}<d(\mathbf x, \mathbf y)\leq 2^{-j+1}\}$, 
and $\Gamma_{Q_1}\cap \Gamma_{Q_2}=\emptyset$ if $Q_1\neq Q_2$.
In fact, for any $(\mathbf x, \mathbf y)\in \partial\Omega_k\times \partial\Omega_k$, there exists a unique $j\in\mathbb Z$ such that
$2^{-j}<d(\mathbf x, \mathbf y)\leq  2^{-j+1}.$
For such $j$, there also exists a unique dyadic cube $Q\in\mathscr D_j$ such that $\mathbf x\in Q$, thus, $(\mathbf x, \mathbf y) \in\Gamma_Q$.
For any $Q_1, Q_2\in\mathscr D$ with $Q_1\neq Q_2$, if $Q_1, Q_2\in \mathscr D_j$, then $Q_1\cap Q_2=\emptyset$. Otherwise, without loss of generality, we may assume that $Q_1\subset Q_2$, then there exist $m, n\in\mathbb Z$ with $n<m$ such that $Q_1\in\mathscr D_m$ and $Q_2\in\mathscr D_n$. Next, we will show that $\Gamma_{Q_1}\cap\Gamma_{Q_2}=\emptyset$. If not, then there is a point $(\mathbf x, \mathbf y) \in \Gamma_{Q_1}\cap\Gamma_{Q_2}$, therefore,
$2^{-m}<d(\mathbf x, \mathbf y)\leq 2^{-m+1}$ and $2^{-n}<d(\mathbf x, \mathbf y)\leq 2^{-n+1}$, but
this is impossible.

\color{black}

By \eqref{decom}, we can see
\begin{align*}
\|b\|^{p}_{B_{p}(\partial\Omega_k)}&=\int_{\partial\Omega_k}\int_{\partial\Omega_k}
\frac{|b(\mathbf x)-b(\mathbf y)|^{p}}{d(\mathbf x, \mathbf y)^2}d\mathbf yd\mathbf x\\
&= \sum_{j\in\mathbb{ Z}}\sum_{Q\in \mathscr{D}_{j}}\int_{Q}\int_{\{\mathbf y\in \partial\Omega_k:2^{-j}<d (\mathbf x, \mathbf y)\leq 2^{-j+1}\}}
\frac{|b(\mathbf x)-b(\mathbf y)|^{p}}{d(\mathbf x, \mathbf y)^2}d\mathbf yd\mathbf x\\
&\leq \sum_{j\in\mathbb{ Z}}\sum_{Q\in \mathscr{D}_{j}}\int_{Q}\int_{\{\mathbf y\in \partial\Omega_k:0<d(\mathbf x, \mathbf y)\leq 3\cdot2^{-j}\}}
\frac{|b(\mathbf x)-b(\mathbf y)|^{p}}{d(\mathbf x, \mathbf y)^2}d\mathbf yd\mathbf x.
\end{align*}
It is clear that $\{\mathbf y\in \partial\Omega_k:0<d(\mathbf x, \mathbf y)\leq 3\cdot2^{-j}\}\subset 6C_dQ$. Then there is $J^{\omega}\in\mathscr{D}^{\omega}$ such that $$6C_dQ\subset \bigcup_{\omega\in\{1,\cdots,\mathpzc T\}}J^{\omega}$$(see for example \cite{HK}). 
\color{black}
Then, we have
\begin{align*}
&\sum_{j\in\mathbb{ Z}}\sum_{Q\in \mathscr{D}_{j}}\int_{Q}\int_{\{\mathbf y\in \partial\Omega_k:0<d(\mathbf x, \mathbf y)\leq 3\cdot2^{-j}\}}
\frac{|b(\mathbf x)-b(\mathbf y)|^{p}}{d(\mathbf x, \mathbf y)^2}d\mathbf yd\mathbf x\\
&\le\sum_{\omega\in\{1,\cdots,\mathpzc T\}}\sum_{J^{\omega}\in\mathscr{D}^{\omega}}
\frac{1}{|J^{\omega}|^2}\int_{J^{\omega}}\int_{J^{\omega}}
|b(\mathbf x)-b(\mathbf y)|^{p}d\mathbf yd\mathbf x\\
&\le 2\sum_{\omega\in\{1,\cdots,\mathpzc T\}}\sum_{J^{\omega}\in\mathscr{D}^{\omega}}
\frac{1}{|J^{\omega}|^2}\int_{J^{\omega}}\int_{J^{\omega}}
|b(\mathbf x)-b_{J^{\omega}}|^{p}d\mathbf xd\mathbf y\\
&=2\sum_{\omega\in\{1,\cdots,\mathpzc T\}}\sum_{J^{\omega}\in\mathscr{D}^{\omega}}
\frac{1}{|J^{\omega}|}\int_{J^{\omega}}
|b(\mathbf x)-E^{\omega}_{h_{J^{\omega}}}(b)(\mathbf x)|^{p}d\mathbf x
=:2I,
\end{align*}
where $h_{J^{\omega}}$ represents that $J^{\omega}\in \mathscr{D}_{h_{J^{\omega}}}$
and $E^{\omega}_h(b)(\mathbf x)=\sum_{J^{\omega}\in\mathscr{D}^\omega_{h}} b_{J^{\omega}}\chi_{J^{\omega}}(\mathbf x)$. 
It suffices to estimate the term $I$.  By noting that $E^{\omega}_h(b)(\mathbf x)\to b(\mathbf x)$ a.e. as $h\to\infty$, following the same idea and techniques as in Lemma 4.7 in \cite{FLL}, we see that 
$I$ is dominated by $$C\sum_{j\in\mathbb{ Z}}\sum_{J^{\omega}\in \mathscr{D}^\omega_{j}}{1\over |J^{\omega}|}\int_{J^{\omega}}|E^{\omega}_{j+1}(b)(\mathbf x)-E^{\omega}_j(b)(\mathbf x)|^pd\mathbf x.$$
Finally we get inequality \eqref{eq-S2.4} and complete the proof of Lemma \ref{lemA1}.
\end{proof}

\subsection{Schatten class estimate: Sufficiency}\label{sec:com}

We note that for sufficiency, we do not see the explicit condition on the critical index $p>4$.
\begin{proposition}\label{schattenlarge2}
Suppose $4<p<\infty$ and $b\in  L^1_{{\rm loc}}(\partial \Omega_k)$. If  $b\in B_{p}(\partial \Omega_k)$, then $[b,{\bf S}]\in S^p$.
\end{proposition}

Before the proof, we need the following lemma  given by Janson and Wolff (\cite[Lemma 1 and Lemma 2]{JW}) in the general measure space setting.

\begin{lemma}[\cite{JW}]\label{weak}
Suppose $(X,\mu)$ is a general measure space, if $p>2$ and $K(x,y)\in  L^{2}(X\times X)$, then the integral operator $T$ associated to the kernel $K(x,y)$ satisfies: 
\begin{align}\label{integral}
\|T\|_{S^{p,\infty}}\leq \|K\|_{L^{p},L^{p^{\prime},\infty}}^{1/2}\|K^{*}\|_{L^{p},L^{p^{\prime},\infty}}^{1/2},
\end{align}
where  $1/p+1/p^{\prime}=1$, $\|\cdot\|_{L^p, L^{p^{\prime},\infty}}$ denotes the mixed-norm:
$$
\|K\|_{L^p,L^{p^{\prime},\infty}}:=\big\|\|K(x,y)\|_{L^p(d\mu(x))}\big\|_{L^{p^{\prime},\infty}(d\mu(y))}.
$$
\end{lemma}

Based on the above estimate for the weak Schatten class $S^{p,\infty}$, we turn to the proof of Proposition \ref {schattenlarge2}.

\begin{proof}[Proof of Proposition \ref {schattenlarge2}]
For every fixed $p>4$, set $1/q=1-2/p$. By \eqref{k-size1} in Theorem \ref{thm CLTW} and by \eqref{mball}, we have that 
\begin{align*}
\left\|S(\mathbf{x}, \mathbf{ y})^{1\over q}\right\|_{L^{\infty},L^{q,\infty}}&=\left\|\frac{1}{d(\mathbf{x}, \mathbf{ y})^ {1\over q}}\right\|_{L^{\infty},L^{q,\infty}}=\sup\limits_{\mathbf{x}\in\partial \Omega_k}\sup\limits_{\alpha>0}\alpha\left|\left\{\mathbf{y}\in \partial \Omega_k:\frac{1}{d(\mathbf{x}, \mathbf{ y})^{1\over q}}>\alpha\right\}\right|^ {1\over q}\\
&=\sup\limits_{\mathbf{x}\in\partial \Omega_k}\sup\limits_{\alpha>0}\alpha \left|B(\mathbf{x},\alpha^{-q})\right|^ {1\over q}
\approx \sup\limits_{\mathbf{x}\in\partial \Omega_k}\sup\limits_{\alpha>0}\alpha\ \big( \alpha^{-q}\big)^{1\over q}\\
&\lesssim1.
\end{align*}
Then by H\"older's inequality for Lorentz spaces, 
\begin{align}\label{verify1}
\left\|(b(\mathbf{x})-b(\mathbf{y}))S(\mathbf{x}, \mathbf{ y})\right\|_{L^p, L^{p^{\prime},\infty}}
&\lesssim \left\|(b(\mathbf{x})-b(\mathbf{y}))S(\mathbf{x}, \mathbf{ y})^{1-{1\over q}}
\right\|_{L^{p},L^{p,\infty}}\left\|S(\mathbf{x}, \mathbf{ y})^{1\over q}\right\|_{L^{\infty},L^{q,\infty}}\nonumber\\
&\lesssim \left\|\frac{b(\mathbf{x})-b(\mathbf{y})}{d(\mathbf{x}, \mathbf{ y})^{1-{1\over q}}}\right\|_{L^{p},L^{p,\infty}}
\lesssim \left\|\frac{b(\mathbf{x})-b(\mathbf{y})}{d(\mathbf{x}, \mathbf{ y})^{2\over p}}\right\|_{L^{p},L^{p}}\nonumber\\
&\approx\|b\|_{B_{p}(\partial \Omega_k)}.
\end{align}
Similarly,
\begin{align}\label{verify2}\left\|(b(\mathbf{x})-b(\mathbf{y}))\overline{S(\mathbf{x}, \mathbf{ y})}\right\|_{L^p, L^{p^{\prime},\infty}}\lesssim \|b\|_{B_{p}(\partial \Omega_k)}.
\end{align}
\color{black}

    Now by  \eqref{verify1}, \eqref{verify2} and \eqref{integral}, we can obtain that
$$\|[b,{\bf S}]\|_{S^{p,\infty}}\leq C\|b\|_{B_{p}(\partial \Omega_k)}.$$ 
Since this inequality holds for all $4<p<\infty$, by the interpolation we have $(S^{p_1,\infty},S^{p_2,\infty})_{\theta_p}=S^{p}$, where $\frac{1-\theta_p}{p_1}+\frac{\theta_p}{p_2}=\frac{1}{p}$. Moreover, since the Besov space $B_p(\partial \Omega_k)$ defined in \eqref{Besov norm} is equivalent to the standard one defined on the space of homogeneous type $(\partial\Omega_k,d,\mu)$ via Litllewood--Paley theory, and hence it has the interpolation  $(B_{p_1},B_{p_2})_{\theta_{p}}=B_{p}$ (c.f. \cite[Theorem 4.1]{MY} and \cite[Theorem 3.1]{Yang}), where $\frac{1-\theta_p}{p_1}+\frac{\theta_p}{p_2}=\frac{1}{p}$. Thus, we obtain that
\begin{align*}
\|[b,{\bf S}]\|_{S^{p}}\leq C\|b\|_{B_{p}(\partial \Omega_k)}.
\end{align*}
This finishes the proof of sufficient condition for the case $4<p<\infty$.
\end{proof}

\subsection{Schatten class estimate: Necessity}\label{sec:com}

\begin{proposition}\label{proS1}
For $p>4$, $b\in {\rm VMO}(\partial\Omega_k)$ with $ \|[b,{\bf S}]\|_{S^p} < \infty$,
we have
$b\in B_{p}(\partial\Omega_k)$ with 
\begin{align*}
\|b\|_{B_{p}(\partial\Omega_k)}\lesssim \|[b,{\bf S}]\|_{S^p}.
\end{align*}
\end{proposition}
\begin{proof}

Below, we consider cubes $Q \in \mathscr{D}$, a fixed dyadic system.

From Lemma \ref{lemA4},
for every dyadic cube $Q\in \mathscr{D}$, there exists another dyadic cube $\hat{Q}$ such that the properties (i) and (ii) in Lemma \ref{lemA4} hold.  Without lost of generality, we may assume that 
$S_1(\mathbf x, \hat{\mathbf x})$ does not change sign for all $(\mathbf x, \hat{\mathbf x}) \in Q \times \hat{Q}$ and
\begin{align*}
|S_1(\mathbf x, \hat{\mathbf x})| \gtrsim \frac{1}{|Q|}.
\end{align*}

Let $m_{b}(\hat{Q})$ be a median value of $b$ over $\hat{Q}$. This means $m_{b}(\hat{Q})$ is a real number such that
\begin{align}\label{median}
E_{1}^{Q}:=\left\{\mathbf y \in Q: b(\mathbf y)<m_{b}(\hat{Q})\right\} \quad\text { and }\quad E_{2}^{Q}:=\left\{\mathbf y \in Q: b(\mathbf y)>m_{b}(\hat{Q})\right\} .
\end{align}

We note that the upper bound $\left|E_{m}^{Q}\right| \leq \frac{1}{2}|Q|$ for $m=1,2$. A median value always exists, but may not be unique.

By noting that $\int_{Q}h^{\epsilon}_{Q}(\mathbf x)d\mathbf x=0$ and using \eqref{median}, a simple calculation gives
\begin{align}\label{Term12}
\left|\int_{Q} b(\mathbf x) h^{\epsilon}_{Q}(\mathbf x) d\mathbf  x\right| &=\left|\int_{Q}\left(b(\mathbf x)-m_{b}(\hat{Q})\right) h^{\epsilon}_{Q}(\mathbf x) d\mathbf  x\right|
 \leq \frac{1}{|Q|^{\frac{1}{2}}} \int_{Q}\left|b(\mathbf x)-m_{b}(\hat{Q})\right| d\mathbf  x \nonumber\\
& \leq \frac{1}{|Q|^{\frac{1}{2}}} \int_{Q \cap E_{1}^{Q}}\left|b(\mathbf x)-m_{b}(\hat{Q})\right| d x+\frac{1}{|Q|^{\frac{1}{2}}}  \int_{Q \cap E_{2}^{Q}}\left|b(\mathbf x)-m_{b}(\hat{Q})\right| d\mathbf  x \nonumber\\
&=: \operatorname{Term}_{1}^{Q}+\operatorname{Term}_{2}^{Q} .
\end{align}
Now we denote
$$
F_{1}^{\hat{Q}}:=\left\{\mathbf y \in \hat{Q}: b(\mathbf y) \geq m_{b}(\hat{Q})\right\}\quad \text { and }\quad F_{2}^{\hat{Q}}:=\left\{\mathbf y \in \hat{Q}: b(\mathbf y) \leq m_{b}(\hat{Q})\right\} .
$$
Then by the definition of $b_{\hat{Q}}$, we have $\left|F_{1}^{\hat{Q}}\right|=\left|F_{2}^{\hat{Q}}\right| \approx|\hat{Q}|$ and $F_{1}^{\hat{Q}} \cup F_{2}^{\hat{Q}}=\hat{Q}$. Note that for $s=1,2$, if $\mathbf x \in E_{s}^{Q}$ and $\mathbf y \in F_{s}^{\hat{Q}}$, then
\begin{align*}
\left|b(\mathbf x)-m_{b}(\hat{Q})\right| &\leq\left|b(\mathbf x)-m_{b}(\hat{Q})\right|+\left|m_{b}(\hat{Q})-b(\mathbf y)\right|
=\left|b(\mathbf x)-m_{b}(\hat{Q})+m_{b}(\hat{Q})-b(\mathbf y)\right|\\
&=|b(\mathbf x)-b(\mathbf y)| .
\end{align*}
Therefore, for $s=1,2$, by using \eqref{Kernel} and by the fact that $|F_{s}^{\hat{Q}}| \approx|Q|$, we have
\begin{align*}
\operatorname{Term}_{s}^{Q} & \lesssim \frac{1}{|Q|^{\frac{1}{2}}} \int_{Q \cap E_{s}^{Q}}|b(\mathbf x)-m_{b}(\hat{Q})| d \mathbf x \frac{|F_{s}^{\hat{Q}}|}{|Q|} 
=\frac{1}{|Q|^{\frac{1}{2}}}  \int_{Q \cap E_{s}^{Q}} \int_{F_{s}^{\hat{Q}}}|b(\mathbf x)-m_{b}(\hat{Q})| \frac{1}{|Q|} d\mathbf  y d\mathbf  x \\
& \lesssim \frac{1}{|Q|^{\frac{1}{2}}} \int_{Q \cap E_{s}^{Q}} \int_{F_{s}^{\hat{Q}}}|b(\mathbf x)-m_{b}(\hat{Q})||S_1(\mathbf x, {\mathbf y})| d\mathbf  y d\mathbf  x \\
& \lesssim \frac{1}{|Q|^{\frac{1}{2}}} \int_{Q \cap E_{s}^{Q}} \int_{F_{s}^{\hat{Q}}}|b(\mathbf x)-b(\mathbf y)||S_1(\mathbf x, {\mathbf y})| d\mathbf y d\mathbf x.
\end{align*}

To continue, by noting that $S_1(\mathbf x, {\mathbf y})$ and $b(\mathbf x)-b(\mathbf y)$ does not change sign for $(\mathbf x, \mathbf y) \in\left(Q \cap E_{s}^{Q}\right) \times F_{s}^{\hat{Q}}$, $s=1,2$, we have that
\begin{align*}
\operatorname{Term}_{s}^{Q} 
&\lesssim\frac{1}{|Q|^{\frac{1}{2}}} \left|\int_{\partial\Omega_k} \int_{\partial \Omega_k}(b(\mathbf x)-b(\mathbf y)) S_1(\mathbf x,\mathbf y) \chi_{F_{s}^{\hat{Q}}}(\mathbf y) d\mathbf  y \chi_{Q \cap E_{s}^{Q}}(\mathbf x) d\mathbf  x\right|\\ 
&\lesssim\frac{1}{|Q|^{\frac{1}{2}}} \left|\int_{\partial\Omega_k} \int_{\partial \Omega_k}(b(\mathbf x)-b(\mathbf y)) \Big(S_1(\mathbf x,\mathbf y) + i S_2(\mathbf x,\mathbf y)\Big)\chi_{F_{s}^{\hat{Q}}}(\mathbf y) d\mathbf  y \chi_{Q \cap E_{s}^{Q}}(\mathbf x) d\mathbf  x\right|\\
&=\frac{1}{|Q|^{\frac{1}{2}}} \left|\int_{\partial\Omega_k} \int_{\partial \Omega_k}(b(\mathbf x)-b(\mathbf y)) S(\mathbf x,\mathbf y) \chi_{F_{s}^{\hat{Q}}}(\mathbf y) d\mathbf  y \chi_{Q \cap E_{s}^{Q}}(\mathbf x) d\mathbf  x\right|.
\end{align*}

Thus, we further obtain that 
\begin{align*}
{\sum_{Q \in \mathscr{D},\epsilon \not \equiv 1}} \left(\frac{|\langle b, h^{\epsilon}_{Q} \rangle| } {|Q|^{\frac{1}{2}}} \right)^p 
&\lesssim{ \sum_{Q \in \mathscr{D}, \epsilon\not\equiv1} \sum_{s=1}^{2}}\Bigg|\int_{\partial\Omega_k} \int_{\partial\Omega_k}(b(\mathbf x)-b(\mathbf y)) S(\mathbf x,\mathbf y) \ \frac{\chi_{F_{s}^{\hat{Q}}}(\mathbf y)}{ |Q|^{1\over2}}\ d\mathbf y\ \frac{ \chi_{Q \cap E_{s}^{Q}}(\mathbf x)}{ |Q|^{1\over2}}\ d \mathbf x\Bigg|^{p} \\
&=:\sum_{Q \in \mathscr{D}, \epsilon\not\equiv1} \sum_{s=1}^{2} \left|\left\langle [b, {\bf S}]  \big(G^{s}_{\hat{Q}} \big), H^{s}_{Q}\right\rangle\right|^{p},
\end{align*}
where
$$
G^{s}_{\hat{Q}}(\mathbf y):=\frac{ \chi_{F_{s}^{\hat{Q}}}(\mathbf y)}{|Q|^{\frac{1}{2}}} \quad \text { and } \quad H^{s}_{Q}(\mathbf x):= \frac{ \chi_{Q \cap E_{s}^{Q}}(\mathbf x)}{|Q|^{\frac{1}{2}}}.
$$
 Then, $\{G^{s}_{\hat{Q}}\}_{\hat{Q}\in\mathscr{D}}$ and $\{H^{s}_{Q}\}_{Q\in\mathscr{D}}$ are NWO sequences for $L^{2}(\partial \Omega_k)$.
It follows from 
Lemma \ref{eq-NWO} that
$$
\|b\|_{ B_{p}(\partial\Omega_k) }\lesssim \|[b,{\bf S}]\|_{S^p} .
$$
The proof of Proposition \ref{proS1} is complete.
\end{proof}

\smallskip

\section{\bf Commutator of Cauchy--Szeg\"o projection: Proof of (2) of Theorem \ref{main thm}} 

In this section we  prove  $(2)$ of Theorem \ref{main thm}. The proof relies on applying Theorem \ref{TaylorForOmega for Intro} about Taylor's expansion to a function $b\in C^2(\partial\Omega_k)$, as well as on the non-degeneracy property of the Cauchy--Szeg\"o kernel proved in Lemma \ref{lemA4}.  

Following Section~\ref{sec:lifting}, we identify $G$ with $\mathbb R^N$. Points in $G$ will then just be vectors  of the form ${\boldsymbol \xi}=(x_1,x_2,x_3,\dots,x_{N-1},x_N)$.
Set $\|{\boldsymbol \xi}\|_G = \tilde{d}_{cc}({\boldsymbol \xi},{\bf 0})$. This is a homogeneous norm on $G$, in the sense that $\| \delta_\lambda ({\boldsymbol \xi})\| = \lambda \|{\boldsymbol \xi}\|$ for every ${\boldsymbol \xi}\in G$ and $\lambda> 0$. Define
\[
   \varphi({\boldsymbol \xi})= \left\{
                \begin{array}{ll}
                  0\qquad \qquad \qquad \quad\hskip.2cm {\rm if} \quad \|{\boldsymbol \xi}\|_G\geq1,\\[5pt]
                  c\exp\big({1\over \|{\boldsymbol \xi}\|_G^{\mathcal Q}-1}\big)\qquad {\rm if} \quad \|{\boldsymbol \xi}\|_G<1,
                 \end{array}
              \right.
  \]
where ${\mathcal Q}$ is the homogeneous dimension of $G$ defined in Section \ref{sec:lifting} and  $c$ is the positive constant such that 
$$\int_{B_{\tilde{d}_{cc}}({\bf 0},1)}\varphi({\boldsymbol \xi})d{\boldsymbol \xi}=1.$$
The integral is with respect to the Haar measure on $G$, which coincides with the Lebesgue measure in $\mathbb R^N$. It is both left and right-invariant.
We also set 
$$ \varphi_\epsilon({\boldsymbol \xi})={1\over \epsilon^{\mathcal Q}} \varphi(\delta_{1\over \epsilon}({\boldsymbol \xi})).$$

For $b\in L^{1}_{loc}(\partial\Omega_k)$, let $\tilde{b}$ be the function on $G$ defined by
$\tilde{b}({\boldsymbol \xi})= b(\Theta^{-1}({\boldsymbol \xi}))$, which is in $L^{1}_{loc}(G)$.
For every ${\boldsymbol \xi}\in G$, and for small $\varepsilon$,
define
$$ \tilde b_\varepsilon({\boldsymbol \xi})= \int_{B_{\tilde{d}_{cc}}(0,\varepsilon)} \tilde{b}({\boldsymbol \xi}{\boldsymbol \eta}^{-1}) \varphi_\epsilon({\boldsymbol \eta}) d{\boldsymbol \eta}=\int_{B_{\tilde{d}_{cc}}({\bf x},\varepsilon)} \tilde{b}({\boldsymbol \eta}) \varphi_\epsilon({\boldsymbol \eta}^{-1}{\boldsymbol \xi})  d{\boldsymbol \eta}.$$
The functions $\tilde b_\varepsilon({\boldsymbol \xi})$ are in $C^\infty(G)$. Moreover, 
$$\tilde{b}_\varepsilon({\boldsymbol \xi})\to \tilde{b}({\boldsymbol \xi}),\qquad {\rm a.e.\ \ as\ \ } \varepsilon\to0.$$
In particular, the functions $b_\varepsilon({\bf x}) = \tilde{b}_\varepsilon(\Theta({\bf x}))$ are smooth on $\partial\Omega_k$ and 
$$b_\varepsilon({\bf x})\to b({\bf x}),\qquad {\rm a.e.\ \ as\ \ } \varepsilon\to0.$$
We set $${b}^{{\bf y}}({{\bf x}}):= {b}({\bf x}-{\bf y}).$$
 \begin{theorem}\label{main thm part 2} Suppose $k\geq2$
and $0<p\leq 4$. If $b\in C_0^\infty(\partial \Omega_k) \subset {\rm VMO}(\partial\Omega_k)$, then $[b,{\bf S}]\in S^p$
if and only if $b$ is a constant. Moreover, if 
$[{b}^{{\bf y}},{\bf S}]\in S^p$ for all $g\in B(o,1)\subset G$ with 
$$\sup_{g\in B(o,1)} \|[{b}^{{\bf y}},{\bf S}]\|_{ S^p} \leq C<\infty,$$
then $b$ is a constant.
\end{theorem}

By Proposition 4.2.17 and Proposition 6.1.1 in \cite{Di}, we see that the size condition of 
the Cauchy--Szeg\"o projection can also be characterized as
\begin{align}\label{k-size1 dcc}
|S(z,t;\bw, s)|
\approx {1\over Vol(B((z,t) ,d_{cc}((z,t),(\bw, s)))}. 
\end{align}

Thus, based on \eqref{k-size1} in Theorem B in the introduction, we see that 
$$ d((z,t),(\bw, s)) \approx Vol(B((z,t) ,d_{cc}((z,t),(\bw, s))).$$

 Also, we denote $\nabla$ be the horizontal gradient of $\partial\Omega_k$ defined by $\nabla f:=(X_{1}f,X_{2}f)$. Then we can show a lower bound for a local pseudo-oscillation of the symbol $b$ in the commutator.
\begin{lemma}\label{lowerbound}
Let $b\in C^{2}(\partial\Omega_k)$. Let $\mathscr{D}$ be a system of dyadic cubes as in Section \ref{sec:dyadic_cubes} such that the reference dyadic points contain the form $(0,0,t_\alpha)$, $\alpha\in \mathscr{A}_\kappa$ for $\kappa\in\mathbb Z$.  Assume that there is a point $\x_{0}\in\partial\Omega_k$ such that $\nabla b(\x_{0})\neq 0$. Then there exist $C>0$, $\varepsilon>0$ and $N>0$ such that if $\kappa>N$, then for any dyadic cube $Q\in \mathscr{D}_{\kappa}$ satisfying $d(\cent(Q),\x_{0})<\varepsilon$, one has $ Q^{\prime}, Q^{\prime\prime}\subset Q$
\begin{align*}
\bigg|\frac{1}{| Q'|}\int_{ Q'}b(\x')d\x'- \frac{1}{| Q''|}\int_{ Q''}b(\x'')d\x''\bigg|\geq C\ell_{d_{cc}}(Q)|\nabla b(\x_{0})|.
\end{align*}
\end{lemma}
\begin{proof}
Since $\nabla b(\x_{0})\neq 0$, we deduce that $(X_1b(\x_{0}), X_2b(\x_{0}))\not=(0,0).$
Choose $\varepsilon$ small enough such that for any $\x$ with $d(\x_0,\x)<3C_d\varepsilon$, $X_ib(\x)$ has the same sign as $X_ib(\x_{0})$ and $|X_ib(\x)|>|X_ib(\x_{0})|/2$, $i=1,2$. Choose $N$ such that when $\kappa>N$, for any dyadic cube $Q\in \mathscr{D}_{\kappa}$, $\ell(Q)\approx 2^{-\kappa}<\varepsilon$.

Denote $ c_Q :=\cent(Q):=\{ c_Q ^{1}, c_Q ^{2},  t_Q\}$ and $\x=(g_1,g_{2},t)$, then by the { Taylor estimates  in Theorem \ref{TaylorForOmega},
$$b(\x)=P_{c_Q}( \x-c_Q)+R(\x, c_Q ),$$
where the remainder term $R(\x, c_Q )$ satisfies
\begin{align*}
|R(\x, c_Q )|\leq Cd_{cc}(\x, c_Q )^2  \sup_{\substack{  d_{cc}(\z, c_Q ) < b  d_{cc}(\x, c_Q ),\\ i,j=1,2 }} | X_iX_j b (\z)|
\end{align*}
and 
$P_{c_Q}( \x-c_Q)$ is the first order ``polynomial'' as constructed in Theorem \ref{TaylorForOmega}.
}

For $Q\in \mathscr{D}_{\kappa}$ with $\kappa>N$, we denote its center $ c_Q =( w_Q, t_Q)\in \mathbb C\times \mathbb R$.  We now consider those $Q\in \mathscr{D}_{\kappa}$ with $\kappa>N$ such that $d( c_Q ,\x_0)<\varepsilon$.  It is clear that $Q$ contains the metric ball $B( c_Q ,  r_Q)$ centered at $ c_Q $ with radius $ r_Q$ comparable to $\ell(Q)$. Based on the proof of Proposition 9.8 in \cite{BGGr3}, we see that for any point $(z,t)\in B( c_Q ,  r_Q)$,
if $r_Q<|w_Q|^{2k+2}/2$ then $|z-w|^4\lesssim r_Q|w|^{2-2k}$; if $r_Q\geq|w_Q|^{2k+2}/2$ then $|z-w|^4\lesssim r_Q^{1+{1-k\over 1+k}}$. We now consider the $d_{cc}$-ball centered at $c_Q$ with radius $\ell_{d_{cc}}(Q)$ whose measure is comparable to the measure of $B( c_Q ,  r_Q)$, which is $r_Q$. Then we recall that from \cite[Prop 3.4.4]{Di} that 
$$ r_Q\approx  \ell_{d_{cc}}(Q)^2 ( \ell_{d_{cc}}(Q)^2|w_Q|^{2k-2}+\ell_{d_{cc}}(Q)^{2k}).$$
Then by considering the two cases $r_Q<|w_Q|^{2k+2}/2$ and $r_Q\geq|w_Q|^{2k+2}/2$ as in the above, we obtain that
$$ |z- w_Q|\lesssim  \ell_{d_{cc}}(Q).  $$
Thus, one can choose $(x'_1,x'_2,  t_Q)$ and $(x''_1,x''_2,  t_Q)$ in $B( c_Q ,  r_Q)$ such that $a_1(x'_1-x''_1)\approx a_2(x'_2-x''_2) \approx \ell_{d_{cc}}(Q)$, where $a_1,a_2\in\{-1,1\}$. Then choosing a large fixed positive number $\kappa_0$, we have that 
there exist  cubes $ Q^{\prime}\in\mathscr{D}_{\kappa +\kappa_0}$, $ Q^{\prime\prime}\in\mathscr{D}_{\kappa+\kappa_0}$ such that $ Q^{\prime}\subset Q$, 
$ Q^{\prime\prime}\subset Q$ and for every $(x'_1,x'_2, t_{ Q'})\in  Q'$, $(x''_1,x''_2, t_{ Q''})\in  Q''$, we have  $X_jb(\x_{0})(x'_{j}-x''_{j})\gtrsim \ell_{d_{cc}}(Q)$ ($j=1,2$). Therefore, by Theorem \ref{TaylorForOmega},
\begin{align*}
&\bigg|\frac{1}{| Q'|}\int_{ Q'}b(\x')d\x'- \frac{1}{| Q''|}\int_{ Q''}b(\x'')d\x''\bigg|\\
&\gtrsim \frac{1}{| Q''| | Q'|}\left|\int_{ Q''}\int_{ Q^{\prime}}{ \big( (x''_1-x'_1)X_1b(\x')+(x''_2-x'_2)X_2b(\x') \big)} d\x'd\x''\right|\\
&\qquad -{1\over | Q''|}\int_{ Q''}{1\over | Q'|}\int_{ Q^{\prime}}|R(\x',\x'')|d\x'd\x'' \\
&\geq {C \ell_{d_{cc}}(Q)}\sum_{j=1}^2|X_jb(\x_0)|-C\ell_{d_{cc}}(Q)^{2} 
\sup_{\substack{  d_{cc}(\z, \x'' ) < b^{2}  d_{cc}(\x',\x''),\\ i,j=1,2 }} | X_iX_j b (\z)|\\
&\geq C\ell_{d_{cc}}(Q)|\nabla b(\x_{0})|,
\end{align*}
where the last inequality holds since we choose $N$ to be a sufficient large constant such that the remainder term can be absorbed by the first term.
\end{proof}

\begin{lemma}\label{const}
Let $b\in C^{2}(\partial\Omega_k)$.  Let $\mathscr{D}$ be a   system of dyadic cubes as in Section 2.2 such that the reference dyadic points contain the form $(0,0,t_\alpha)$, $\alpha\in \mathscr{A}_\kappa$ for $\kappa\in\mathbb Z$. Suppose that 
\begin{align}\label{b norm finite}
\left\|\left\{ b_{ Q''}-b_{ Q^{\prime}}\right\}_{T\in \mathscr{D}}\right\|_{\ell^{4}}<+\infty,
\end{align}
then $b$ is a constant. 
\end{lemma}
\begin{proof}
We prove by contradiction. Assume that there is a point $\x_{0}\in\partial\Omega_k\backslash\{(0,0,t)\in\partial\Omega_k\}$ such that $\nabla b(\x_{0})\neq 0$.

We now consider a small number $\varepsilon>0$ and a large number $N>0$ with $2^{-N}<\varepsilon$, and then consider all cubes $Q\in \mathscr{D}_{\kappa}$ with $\kappa >N$ and with center $ c_Q =(w_Q,t_Q)$, satisfying $d( c_Q ,\x_{0})<\varepsilon$. Then we know that 
$\mu(Q)\approx 2^{-\kappa}$
and that for each $\kappa>N$, the number of such $Q$ is around $\varepsilon 2^{\kappa}$.
It is clear that $Q$ contains the metric ball $B( c_Q ,  r_Q)$ centered at $ c_Q $ with radius $ r_Q$ comparable to $\ell(Q)$. We now consider again the $d_{cc}$-ball centered at $c_Q$ with radius $\ell_{d_{cc}}(Q)$ whose measure is comparable to the measure of $B( c_Q ,  r_Q)$, which is $2^{-\kappa}$.

We now consider the case $\x_0=(x_0,y_0,t_0)$ with $x_0^2+y_0^2\not=0$. 
Without lost of generality, we have $x_0^2+y_0^2=1$. Then we choose $\varepsilon<1/2$. Based on the proof of Proposition 9.8 in \cite{BGGr3}, it is clear that $|w_Q-(x_0,y_0)|^4\lesssim \varepsilon$. This gives $|w_Q|\gtrsim 1-\varepsilon^{1\over4}$. 
Then we will only have the case that $|w_Q|\gtrsim \ell_{d_{cc}}(Q)$. Hence,
$2^{-\kappa} \approx \delta^4.$
Thus, $\delta\approx \ell(Q)^{1\over 4}\approx 2^{ -{\kappa\over 4}}.$

Let $b\in C^{2}(\partial\Omega_k)$ satisfying \eqref{b norm finite}.
Suppose that $b$ is not a constant.
Then there exists a point $\x_{0}\in \partial\Omega_k$ such that $\nabla b(\x_{0})\neq 0$. By Lemma \ref{lowerbound}, there exist  $\varepsilon>0$ and $N>0$ such that if $\kappa>N$, then for any cube $Q\in \mathscr{D}_{\kappa}$ satisfying $d(\cent(Q),\x_{0})<\varepsilon$,
\begin{align*}
\big| b_{ Q''}-b_{ Q^{\prime}}\big|\geq C\ell_{d_{cc}}(Q)|\nabla b(\x_{0})|.
\end{align*}
Note that for $\kappa>N$, the number of $Q\in \mathscr{D}_{\kappa}$ and $d(\cent(Q),\x_{0})<\varepsilon$ is at least  {$c 2^{\kappa}$}. 
Therefore, we have the following estimate.

If $\x_0=(x_0,y_0,t_0)$ where $x_0^2+y_0^2>0$, then we have
\begin{align*}
\left\|\left\{ b_{ Q''}-b_{ Q^{\prime}}\right\}_{ Q\in\mathscr{D}}\right\|_{\ell^{4}}
&\gtrsim  \Bigg(\sum_{\kappa=N+1}^{\infty}\sum_{\substack{Q\in \mathscr{D}_{\kappa}\\ Q:\ d(\cent(Q),\x_{0})<\varepsilon}} \ell_{d_{cc}}(Q)^{4}|\nabla b(\x_{0})|^{4}\Bigg)^{1\over 4}\\
&\gtrsim |\nabla b(\x_{0})|\Bigg(\sum_{\kappa=N+1}^{\infty}\sum_{\substack{Q\in \mathscr{D}_{\kappa}\\ Q:\ d(\cent(Q),\x_{0})<\varepsilon}} 2^{-\kappa }\Bigg)^{1\over 4}\\
&\gtrsim |\nabla b(\x_{0})|\Bigg(\sum_{\kappa=N+1}^{\infty} 2^\kappa 2^{-\kappa }\Bigg)^{1\over 4}\\
&=+\infty.
\end{align*}
This is also contradicted to the inequality \eqref{b norm finite}. 

Thus, $\nabla b(\x_{0})= 0$ for all $\x_{0}\in\partial\Omega_k\backslash\{(0,0,t)\in\partial\Omega_k\}$, and hence $b$ is a constant on $\x_{0}\in\partial\Omega_k\backslash\{(0,0,t)\in\partial\Omega_k\}$. Since the measure of $\{(0,0,t)\in\partial\Omega_k\}$ is zero, and $b$ is continuous, we know that $b$ is a constant.

Therefore, the proof of Lemma \ref{const} is complete.
\end{proof}

Let $\mathscr{D}_{\kappa}$ be the collection of dyadic cubes at level $k$ as in Section 2.1. We define the conditional expectation of a locally integrable function $f$ on $\partial\Omega_k$ with respect to the increasing family of $\sigma-$algebras $\sigma(\mathscr{D}_{\kappa})$ by the expression: $$E_{\kappa}(f)(\x)=\sum_{Q\in \mathscr{D}_{\kappa}}(f)_{Q}\chi_{Q}(\x),\ \x\in\partial\Omega_k,$$
where we denote $(f)_{Q}$ be the average of $f$ over $Q$, that is, $(f)_{Q}:={1\over |Q|}\int_{Q}f(\x)d\x$.

For $Q\in \mathscr{D}_{\kappa}$, we let $h_{Q}^{1}$, $h_{Q}^{2},\ldots, h_{Q}^{M_Q-1}$ be a family of Haar functions constructed in \cite{KLPW}. 
Next, we choose $h_{Q}$ among these Haar functions such that $\left|\int_{Q}b(\x)h_{Q}^{\epsilon}(\x)\,d\x\right|$ is maximal.

Note that the function $(E_{\kappa+1}(b)(\x)-E_{\kappa}(b)(\x))\chi_{Q}(\x)$ is a sum of $M_Q$ Haar functions. That is, we are in a finite dimensional setting and all $L^p$-spaces have comparable norms. So
we have that
\begin{align}\label{tttt1}
\left({1\over |Q|}\int_{Q}|E_{\kappa+1}(b)(\x)-E_{\kappa}(b)(\x)|^{p}\,d\x\right)^{1/p}
&\leq C  |T|^{-1/2}\left|\int_{Q}b(\x)h_{Q}(\x)\,d\x\right|,
\end{align}
where $C$ is a constant only depending on $p$ and $n$.

\begin{lemma}\label{lem approximation 1}
Let all the notation be the same as above.
Let $b\in L^{1}_{loc}(\partial\Omega_k)$, $\kappa_0$ be a fixed positive integer. 
Suppose that
\begin{align}\label{b norm finite}
&\sup_{{\bf y}: d_{cc}({\bf y},0)\leq1} \bigg\| \bigg\{ {1\over |Q|}\int_Q{1\over |Q|}\int_Q| E_{\kappa+\kappa_0}({b}^{{\bf y}})(\x')-E_{\kappa+\kappa_0}({b}^{{\bf y}})(\x'') | d\x'd\x''  \bigg\}_{(\kappa,Q):\ \kappa\in\mathbb Z, Q\in\mathscr{D}_\kappa} \bigg\|_{\ell^4}\\
&<+\infty,\nonumber
\end{align}
then $b$ is a constant. 
\end{lemma}
\begin{proof}
We now claim that for all small enough $\varepsilon$, $b_\varepsilon(\xi)$ is a constant.

If $b_\epsilon$ is not a constant, then there exists a point $\x_{0}\in \partial\Omega_k$ such that $\nabla b(\x_{0})\neq 0$.  Note that
\begin{align*}
&\bigg\| \bigg\{ {1\over |Q|}\int_Q{1\over |Q|}\int_Q| E_{\kappa+\kappa_0}(b_\varepsilon)(\x')-E_{\kappa+\kappa_0}(b_\varepsilon)(\x'') | d\x'd\x''  \bigg\}_{(\kappa,Q):\ \kappa\in\mathbb Z, Q\in\mathscr{D}_\kappa} \bigg\|_{\ell^4}\\
&\leq C \sup_{{\bf y}: d_{cc}({\bf y},0)\leq1}  \bigg\| \bigg\{ {1\over |Q|}\int_Q{1\over |Q|}\int_Q| E_{\kappa+\kappa_0}({b}^{{\bf y}})(\x')-E_{\kappa+\kappa_0}({b}^{{\bf y}})(\x'') | d\x'd\x''  \bigg\}_{(\kappa,Q):\ \kappa\in\mathbb Z, Q\in\mathscr{D}_\kappa} \bigg\|_{\ell^4}.
\end{align*}
Hence, 
\eqref{b norm finite} implies that 
\begin{align*}
&\bigg\| \bigg\{ {1\over |Q|}\int_Q{1\over |Q|}\int_Q| E_{\kappa+\kappa_0}(b_\varepsilon)(\x')-E_{\kappa+\kappa_0}(b_\varepsilon)(\x'') | d\x'd\x''  \bigg\}_{(\kappa,Q):\ \kappa\in\mathbb Z, Q\in\mathscr{D}_\kappa} \bigg\|_{\ell^4}<+\infty.
\end{align*}
Moreover, we have
\begin{align*}
&\bigg\| \bigg\{ {1\over |Q|}\int_Q{1\over |Q|}\int_Q| E_{\kappa+\kappa_0}(b_\varepsilon)(\x')-E_{\kappa+\kappa_0}(b_\varepsilon)(\x'') | d\x'd\x''  \bigg\}_{(\kappa,Q):\ \kappa\in\mathbb Z, Q\in\mathscr{D}_\kappa} \bigg\|_{\ell^4}\\
&>\bigg\| \bigg\{ {1\over |Q|}\int_Q{1\over |Q|}\int_Q| E_{\kappa+\kappa_0}(b_\varepsilon)(\x')-E_{\kappa+\kappa_0}(b_\varepsilon)(\x'') | d\x'd\x''  \bigg\}_{\substack{(\kappa,Q):\ \kappa\in\mathbb Z, \kappa>N, \\  Q\in\mathscr{D}_\kappa, d(\cent(Q),\x_{0})<\varepsilon }} \bigg\|_{\ell^4}\\
&\geq C\left\|\left\{ b_{ Q''}-b_{ Q^{\prime}}\right\}_{ Q\in\mathscr{D}}\right\|_{\ell^{4}}\\
&=+\infty
\end{align*}
by the argument in Lemma \ref{const}.

Hence, we have a contradiction. This yields that for every small enough $\varepsilon>0$, $b_\varepsilon(\xi)$ is a constant for all $\xi$ and  for any small $\varepsilon$. Since $b_\varepsilon\to b$  {\rm a.e.\  as\ } $\varepsilon\to0$, we see that $b$ is a constant.
\end{proof}

\begin{lemma}\label{lem approximation 2}
Let all the notation be the same as above.
Let $b\in {\rm VMO}(\partial\Omega_k)$ with $[b,{\bf S}]\in S^4$.
Then for any ${\bf y}\in\partial\Omega_k$ with $d_{cc}({\bf y},0)\leq1$, 
\begin{align}\label{b norm by Sp norm}
& \bigg\| \bigg\{ {1\over |Q|}\int_Q{1\over |Q|}\int_Q| E_{\kappa+\kappa_0}({b}^{{\bf y}})(\x')-E_{\kappa+\kappa_0}({b}^{{\bf y}})(\x'') | d\x'd\x''  \bigg\}_{(\kappa,Q):\ \kappa\in\mathbb Z, Q\in\mathscr{D}_\kappa} \bigg\|_{\ell^4}\\
&\leq C\| [{b}^{{\bf y}},{\bf S}]\|_{S^4},\nonumber
\end{align}
where the positive constant $C$ is independent of $b$ and $g$. 
\end{lemma}
\begin{proof}
It suffices to show that for any fixed ${\bf y}\in\partial\Omega_k$ with $d_{cc}({\bf y},0)\leq1$,
\begin{align}
& \sum_\kappa 2^\kappa \| E_{\kappa+1}({b}^{{\bf y}})-E_{\kappa}({b}^{{\bf y}})\|_{L^4(\partial\Omega_k)}^4\leq C \|[{b}^{{\bf y}},{\bf S}]\|_{ S^4},
\end{align}
where the positive constant $C$ is independent of $b$ and ${\bf y}$. 

We now fix ${\bf y}\in\partial\Omega_k$ with $d_{cc}({\bf y},0)\leq1$, and then follow the proof in Proposition \ref{proS1}. By \eqref{tttt1}, we have
\begin{align}\label{comcom1}
2^\kappa\int_{\partial\Omega_k}|E_{\kappa+1}({b}^{{\bf y}})(\x)-E_{\kappa}({b}^{{\bf y}})(\x)|^{4}d\x
&=\sum_{Q\in \mathscr{D}_{\kappa}}{1\over |Q|}\int_{Q}|E_{\kappa+1}({b}^{{\bf y}})(\x)-E_{\kappa}({b}^{{\bf y}})(\x)|^{4}d\x\nonumber\\
&\leq C\sum_{Q\in \mathscr{D}_{\kappa}}|Q|^{-2}\left|\int_{Q}{b}^{{\bf y}}(\x)h_{Q}(\x)d\x\right|^{4}.
\end{align}
To continue, for any $Q\in \mathscr{D}_{\kappa}$, let $\hat{Q}$ be the cube chosen in Lemma \ref{lemA4}, then without lost of generality, we may assume that 
$S_1(\x, \hat{\x})$ does not change sign for all $(\x, \hat{\x})\in Q \times \hat{Q}$ and
\begin{align}\label{lower for S4}
|S_1(\x, \hat{\x})| \gtrsim \frac{1}{|Q|}.
\end{align}
 Also, let $ m_{\hat{Q}}({b}^{{\bf y}})$ be a median value of ${b}^{{\bf y}}$ over $\hat{Q}$, as defined in \eqref{median}.

Next we denote
\begin{align*}
E_{1}^{Q}&:=\left\{\x\in Q:\ {b}^{{\bf y}}(\x) <  m_{\hat{Q}}({b}^{{\bf y}})\right\}\ \ {\rm and}\ \
E_{2}^{Q}:=\left\{\x\in Q:\ {b}^{{\bf y}}(\x)\geq m_{\hat{Q}}({b}^{{\bf y}})\right\},\\
F_{1}^{\hat Q}&:=\left\{{\hat \x}\in \hat{Q}:{b}^{{\bf y}}({\hat \x})\geq m_{\hat{Q}}({b}^{{\bf y}})\right\}\ \ {\rm and}\ \
F_{2}^{\hat Q}:=\left\{{\hat \x}\in \hat{Q}:{b}^{{\bf y}}({\hat \x})\leq m_{\hat{Q}}({b}^{{\bf y}})\right\}.
\end{align*}
Then by the definition of $ m_{\hat{Q}}({b}^{{\bf y}})$, we have $|F_{1}^{\hat Q}|=|F_{2}^{\hat Q}|\sim|\hat{Q}|$ and $F_{1}^{\hat Q}\cup F_{2}^{\hat Q}=\hat{Q}$. Note that for $s=1,2$, if $\x\in E_{s}^{Q}$ and $ \hat \x\in F_{s}^{\hat Q}$, then again we have
\begin{align*}
\left|{b}^{{\bf y}}(\x)- m_{\hat{Q}}({b}^{{\bf y}})\right|
&\leq
\left|{b}^{{\bf y}}({\hat \x})-{b}^{{\bf y}}(\x)\right|.
\end{align*}

Then by using the same calculation as in \eqref{Term12} we see that 
\begin{align}\label{comcom2}
&|Q|^{-1/2}\left|\int_{Q}{b}^{{\bf y}}(\x)h_{Q}(\x)d\x\right|\nonumber\\
&\leq \frac{1}{|Q|}\int_{Q\cap E_{1}^{Q}}\left|{b}^{{\bf y}}(\x)- m_{\hat{Q}}({b}^{{\bf y}})\right|d\x+ \frac{1}{|Q|}\int_{Q\cap E_{2}^{Q}}\left|{b}^{{\bf y}}(\x)- m_{\hat{Q}}({b}^{{\bf y}})\right|d\x\nonumber\\
&=:{\rm I}_{1}^{Q}+{\rm I}_{2}^{Q}.
\end{align}

Therefore, for $ s=1,2$,
\begin{align}\label{haha}
{\rm I}_{s}^{Q}&\lesssim \frac{1}{|Q|}\int_{Q\cap E_{s}^{Q}}\left|{b}^{{\bf y}}(\x)- m_{\hat{Q}}({b}^{{\bf y}})\right|d\x\frac{|F_{s}^{\hat Q}|}{|Q|}\nonumber\\
&\lesssim \frac{1}{|Q|}\int_{Q\cap E_{s}^{Q}}\int_{F_{s}^{\hat Q}}\left|{b}^{{\bf y}}(\x)- m_{\hat{Q}}({b}^{{\bf y}})\right|\left|S_1(\x, \hat{\x})\right|d{\hat \x}d\x\nonumber\\
&\lesssim \frac{1}{|Q|}\int_{Q\cap E_{s}^{Q}}\int_{F_{s}^{\hat Q}}\left|{b}^{{\bf y}}({\hat \x})-{b}^{{\bf y}}(\x)\right|\left|S_1(\x, \hat{\x})\right|d{\hat \x}d\x\nonumber\\
&=\frac{1}{|Q|}\left|\int_{Q\cap E_{s}^{Q}}\int_{F_{s}^{\hat Q}}({b}^{{\bf y}}({\hat \x})-{b}^{{\bf y}}(\x))S_1(\x, \hat{\x})\ d{\hat \x}d\x\right|,
\end{align}
where in the last equality we used the fact that $S_1(\x, \hat{\x})$ and ${b}^{{\bf y}}({\hat \x})-{b}^{{\bf y}}(\x)$ do not  change sign for $(\x,{\hat \x})\in (P_{i}\cap E_{s}^{Q})\times F_{s}^{Q}$, $s=1,2$. This, in combination with the inequalities \eqref{comcom1} and \eqref{comcom2}, implies that
\begin{align}\label{eee ortho S norm}
2^\kappa\int_{\partial\Omega_k} |E_{\kappa+1}({b}^{{\bf y}})(\x)-E_{\kappa}({b}^{{\bf y}})(\x)|^{4}d\x
& \lesssim \sum_{Q\in \mathscr{D}_{\kappa}}|Q|^{-2}\left|\int_{Q}{b}^{{\bf y}}(\x)h_{Q}(\x)d\x\right|^{4}\nonumber
 \lesssim  \sum_{s=1}^{2}\sum_{Q\in \mathscr{D}_{\kappa}}\left|{\rm I}_{s}^{Q}\right|^{4}\nonumber
\\
&\lesssim  \sum_{s=1}^{2}\sum_{Q\in \mathscr{D}_{\kappa}}\left(\sum_{i=1}^{M_Q}\left|\left\langle[{b}^{{\bf y}},{\bf S}] \frac{|P_{i}|^{1/2}
\chi_{F_{s}^{Q}}}{|Q|},\frac{\chi_{P_i\cap E_{s}^{Q}}}{|P_{i}|^{1/2}}\right\rangle\right|\right)^{4}.
\end{align}
Note that $e_Q:=\frac{|P_{i}|^{1/2}
\chi_{F_{s}^{Q}}}{|Q|} \subset \hat Q$ and $f_Q:=\frac{\chi_{P_i\cap E_{s}^{Q}}}{|P_{i}|^{1/2}} \subset Q$.
Sum this last inequality over $ \kappa\in \mathbb Z $, and appeal  to Lemma \ref{eq-NWO} to conclude \eqref{comcom1}.

The proof is complete.
\end{proof}

We now show Theorem \ref{main thm part 2}.

\begin{proof}[Proof of Theorem \ref{main thm part 2}]
Suppose $k\geq2$
and $0<p\leq 4$. If $b$ is a constant, then $[b,{\bf S}]$ is a zero operator, which is certainly in $S^p$.
Conversely, if $b\in C_0^\infty(\partial \Omega_k) \subset {\rm VMO}(\partial\Omega_k)$ with $[b,{\bf S}]\in S^p$, then by Lemmas \ref{lem approximation 1} and \ref{lem approximation 2} with ${\bf y}=0\in \partial\Omega_k$, we obtain that  $b$ is a constant.

Moreover, if there is a positive constant $C$ such that 
$$\sup_{ {\bf y}\in\partial\Omega_k:\ d_{cc}({\bf y},0)\leq1 } \|[{b}^{{\bf y}},{\bf S}]\|_{ S^p} \leq C<\infty,$$
then by Lemmas \ref{lem approximation 1} and \ref{lem approximation 2}, we see that $b$ is a constant.
\end{proof}

\bigskip

\bigskip
\bigskip

{\bf Acknowledgement:} Der-Chen Chang is partially supported by an NSF grant DMS-1408839 and a McDevitt Endowment Fund at Georgetown University. Ji Li and Alessandro Ottazzi are supported by the Australian Research Council (ARC) through the research grant DP220100285. Qingyan Wu is supported by the National Science Foundation of China (grant nos. 12171221 and 12071197), the Natural Science Foundation of Shandong Province (grant nos. ZR2021MA031 and 2020KJI002).

\end{document}